\documentclass[12pt]{article}
\usepackage{amsmath,amsfonts,amssymb,amsthm,amscd}
\usepackage{xcolor}
\title{On definable groups in real closed fields with a generic derivation, and related structures}
\date{\today}
\author{Ya'acov Peterzil\thanks{Partially supported by  ISF grant 290/19}\\{University of Haifa} \and Anand Pillay \thanks{Partially supported by NSF grants DMS-1665035, DMS-1760212, and DMS-2054271}\\{University of Notre Dame}\and Fran\c{c}oise Point\thanks{Partially supported by Fonds de la Recherche Scientifique-FNRS-FRS}\\{University of Mons}}

\newtheorem{Theorem}{Theorem}[section]
\newtheorem{Proposition}[Theorem]{Proposition}
\newtheorem{Definition}[Theorem]{Definition}
\newtheorem{Remark}[Theorem]{Remark}
\newtheorem{Lemma}[Theorem]{Lemma}
\newtheorem{Corollary}[Theorem]{Corollary}
\newtheorem{Fact}[Theorem]{Fact}

\begin{document}
\maketitle

\begin{abstract}
We study finite-dimensional groups definable in models of the theory $RCF_{\partial}$ of real closed fields with a generic derivation (also known as $CODF$, the theory of closed ordered differential field \cite{Singer}). 
We prove that any such group $\Gamma$ definably embeds in a semialgebraic group $G$. 

We explain how our methods work in the general context of strongly model complete theories $T$ of large ``geometric" fields with a generic derivation, which includes the cases where $T$ is the theory of pseudofinite fields and $T = Th({\mathbb Q}_{p})$.  We also give a general theorem on recovering a definable group from generic data in the context of geometric theories.

Finally we extend the methods to $o$-minimal theories with a generic derivation, due to Fornasiero and Kaplan,  \cite{F-K}, and open theories of topological fields with a generic derivation, due to Cubides-Kovacsics and the third author, \cite{K-P}.

\end{abstract}

\section{Introduction}
The aim is to describe finite-dimensional definable groups in models of the theory $RCF_{\partial}$, also known as closed ordered differential fields.   In particular  any such group $G$ definably embeds in a group $H$ which is semialgebraic, namely definable in the restriction to the field language.

The case of infinite-dimensional groups is of course also interesting and will be considered in future work.

The proof goes through various stages.  First, in Section 3,  from a finite-dimensional group $\Gamma$ we obtain a ``generically defined" group in $RCF$, which we may call a semialgebraic pre-group.  Then, in Section 4,  we construct, from the generically defined group,  a definable group $G$ in $RCF$. Finally we show that $\Gamma$ definably embeds in $G$. 

 

Although the results and proofs are first presented for groups definable in models of $RCF_{\partial}$, they go through in various ways to more general or related contexts, as mentioned in the abstract.  Details are given below and  in subsequent sections. 

One of the reasons that we start with the  theory $RCF_{\partial}$ is that it is  fairly well-known and of interest to several researchers. 

In a paper in preparation we will adapt the theory of algebraic $D$-groups to some of the more general settings considered in this paper, obtaining in particular the notion of a Nash $D$-group (a Nash group over a real closed field with a derivation, equipped with a ``connection" respecting the group structure).  This will yield tighter descriptions of $\Gamma$ and $G$.


The theory $CODF$ of closed ordered differential fields was introduced by  Michael Singer \cite{Singer} as an expansion of the theory of real closed ordered fields obtained by adjoining a derivation $\partial$ and with axioms analogous to L. Blum's one variable axioms for $DCF_{0}$, bearing in mind the ordering. Singer worked in what he called the {\em language of ordered fields}, and his axioms for $CODF$ consisted of the axioms for real closed (ordered) fields, the axioms for ordered differential fields (i.e. fields equipped with an ordering and a derivation), as well as the following:

Let $f,g_{1},..,g_{m}$ be differential polynomials in one variable $y$, such that $f$ has order $n$ which is $\geq$  the  order of each  $g_{i}$. Suppose that, viewing $f$ and the $g_{i}$ as ordinary polynomials of at most $n+1$ variables, there is $(c_{0},..,c_{n})$ such that $f(c_{0},..,c_{n}) = 0$ and $g_{i}(c_{0},..,c_{n}) > 0$, for $i = 1,..,m$ and ${\partial}f /{\partial}y^{(n)}(c_{0},..,c_{n}) \neq 0$. Then,  viewing $f, g_i$ back as differential polynomials in one indeterminate, there is $z$ such that $f(z) = 0$ and $g_{i}(z) > 0$ for $i=1,..,m$. 

Singer proved in \cite{Singer} that $CODF$ is the model completion in the language of ordered fields of the theory $ODF$ of ordered differential fields, namely that $CODF$ is the model companion of $ODF$ and that whenever $K$ is a model of $CODF$ and $K_{0}\subseteq K$ is a substructure (so model of $ODF$) then $CODF$ together with the diagram of $K_{0}$ is complete. This precisely means that $CODF$ has quantifier elimination in the language of ordered differential fields. Let us stress here that by an ordered differential field, we mean simply a field equipped with both an ordering and a derivation where no additional relation between the ordering and derivation is assumed.  

It is more usual in the model theory of fields to work in the language of rings rather than fields (so in the language of ordered rings rather than ordered fields, when studying ordered fields). So we will work in the language of ordered rings,  and remark that it easily follows from Singer's results that $CODF$  also has quantifier elimination in this language.

Hence we can  redefine $CODF$ as follows:

\begin{Fact} (i) The (universal) theory of ordered differential integral domains, in the language of ordered differential rings,  has a model companion, which we call $CODF$.
\newline
(ii) $CODF$ is complete, with quantifier elimination (in the language of ordered differential rings). 
\end{Fact}

 We could work just in the differential ring language, from which point of view $CODF$ is the model companion of the theory of formally real differential fields.  This is placed in a more general context by Tressl in his work on ``uniform model companions"  \cite{Tressl} and is relevant to some of our subsequent generalizations, so we give a few details.  The reader can see \cite{Tressl} for more references. 

From now on  $L$ will be the language of rings and $L_{\partial}$ the language of differential rings. We work throughout in characteristic $0$. 
A field $K$ is said to be {\em large} \cite{Pop} if any irreducible curve over $K$ with a smooth $K$-point has infinitely many $K$-points. The class of large fields is axiomatizable, and Tressl points out that if $T$ is a model complete theory of large fields in the language $L$, then the theory $T\cup \{{\partial}$ is a derivation$\}$ has a model companion, which we will call $T_{\partial}$. Actually Tressl works with several commuting derivations $\partial_{1},...,\partial_{m}$ but we are here just concerned with the case $m=1$. An important aspect of Tressl's work is that axioms can be given for $T_{\partial}$ which are uniform across $T$.  These can be formulated as so-called ``geometric axioms" (see Fact 2.8 and Remark 2.9 (2) of \cite{Leon-Sanchez-Pillay}) and will be discussed later.   But the connection with  $CODF$, is that real closed fields are large, and that the theory $RCF$ of real closed fields in the ring language is model-complete, whereby Tressl's theory gives us the model companion $RCF_{\partial}$ of $RCF \cup\{{\partial}$ is a derivation$\}$.  (8.2)(ii) of \cite{Tressl}, says:
\begin{Fact} $CODF$, when formulated in the language of differential rings coincides with $RCF_{\partial}$. 
\end{Fact}

From now on we will use the expression $RCF_{\partial}$, instead of $CODF$, 
as it will be easier notation-wise to adjust to the variants $T_{\partial}$.  Even though one works in the differential ring language, there is of course a unique ordering which is definable, so it makes sense to talk about semialgebraic sets.
We may often say ``by quantifier elimination" when we mean quantifier elimination in the ordered differential ring language or ordered ring language, depending on the context. We hope there is no confusion for the reader.  

Our main result is:
\begin{Theorem} Let $\Gamma$ be a finite-dimensional group definable in a model $(K,\partial)$ of $RCF_{\partial}$. Then there is a group $G$ definable in the real closed field $K$, and a definable (in $(K,\partial)${\color{blue})} embedding of $\Gamma$ in $G$.
\end{Theorem}

See Theorem 4.1 and Remark 4.8 in Section 4.

The corresponding result for {\em connected} groups of finite Morley rank definable in differentially closed fields is Lemma 1.1 of \cite{Pillay-differential-algebraic-groups}.  The methods there are stability-theoretic, and a main point of the current paper is to adapt the ideas to unstable situations where there are {\em many} ``types of maximal dimension". 
The corresponding result for arbitrary (possibly infinite Morley rank) groups definable in differentially closed fields, namely that they are embeddable in algebraic groups, appears in \cite{Pillay-foundational}, and the methods there will {\em not} adapt to the unstable contexts

We now want to state formally the generalizations which we will also prove. We will mention the contexts, state the results, then say something about the notation. There are three overlapping contexts.
In context (i), discussed in Section 5.2, $T$ is a theory of large geometric fields in the language $L$ of rings which is strongly model complete. $T_{\partial}$ is the model companion of $T\cup\{\partial$ is a derivation\}, in the language $L_{\partial} = L\cup\{\partial\}$. 

In context (ii), discussed in Section 6.1, $T$ is a complete, model-complete expansion of the theory of real closed fields, in a language $L$ which {\em  expands} the language of ordered rings. $T_{\partial}$ is the model companion of $T\cup\{\partial$ is a compatible derivation\}  in the language $L_{\partial} = L\cup\{\partial\}$.

In context (iii), discussed in Section 6.2, $T$ is an ``open theory of large topological fields" of characteristic $0$ in a language $L$ expanding the language of rings by constant and relation symbols. $T_{\partial}$ is the model companion of $T\cup\{\partial$ is a derivation\} in the language $L_{\partial} = L\cup \{\partial\}$. 

So note that the language $L$ (so also $L_{\partial}$)  may be different in these three contexts, and the existence of the relevant model companion $T_{\partial}$ is part of the literature. 
\begin{Theorem}
In each of the contexts (i), (ii), (iii), let $(K,\partial)$ be a model of $T_{\partial}$ (and $K$ a model of $T$ in the language $L$). Let $\Gamma$ be a group definable in the structure $(K,\partial)$. Then there is a definable (in $(K,\partial)$) embedding of $\Gamma$ in a group $G$ which is definable in the $L$-structure $K$. 
\end{Theorem}
For context (i), geometric fields were introduced in \cite{Hrushovski-Pillay-groupslocalfields}, and strongly model complete is also known as ``almost QE" and appears in \cite{BCPP}.  Anyway, see Theorem 5.11. As mentioned in the abstract,  context (ii) appears in \cite{F-K} and details will be given in Section 6.1. See Theorem 6.8. Likewise context (iii) appears in \cite{K-P} and more details appear in Section 6.2.  See Theorem 6.12.  Each of these contexts subsumes the special case of $RCF_{\partial}$ and more examples will also be given below. 

Finally in this introduction, we mention another generalization of one aspect of the proof of Theorem 1.3 and which plays a role in the proofs of Theorem 1.4. This is the construction of a definable group from a generically definable group, in the context of an arbitrary geometric structure, and extends the ``group chunk" methods from \cite{Weil}. 
Geometric structures were again introduced in \cite{Hrushovski-Pillay-groupslocalfields}. 
\begin{Theorem} Let $M$ be a geometric structure and let $(X, F, g)$ be a generically defined group on the  definable set $X$ (all in the sense of $M$). Then there is a definable map $h$ of $X$ into a definable group $G$, with various properties, including that the images under $h$ of the generically defined multiplication $F$ and inversion $g$ on $X$ are multilplication and inversion on $G$. Also  $h(X)$ is ``large" in $G$. 
\end{Theorem}

See Section 5.1 for the precise definitions and Theorem 5.4 there.  

Our model-theoretic and differential-algebraic notation is standard. But one can see the volume \cite{Bouscaren} for a lot of background on these topics. In particular the article by the second author on the model theory of algebraically closed fields gives an account of algebraic varieties which is relevant to the current paper. 

We will typically be working in one-sorted structures, mainly fields equipped with a derivation.  
In this context ${\bar a}, {\bar b}$,... will denote finite tuples from the ``home sort" and likewise for finite tuples of variables. But nevertheless when there is no ambiguity, 
we may let $a,b,..$ denote finite tuples from the universe, and $x,y,..$ finite tuples of variables. 

Let us say a few words about the possible interest or relevance of work in this paper. There may be a sense that when working with ordered or valued fields one should really posit some additional compatabilities between the derivation and ordering (or valuation), as for example is done in asymptotic differential algebra. Nevertheless there is also interest and work in ``real differential algebra" and even ``real differential Galois theory" where no such additional compatibilities are assumed. Likewise in the $p$-adic case. See for example \cite{CHvdP}. 
So establishing ``universal domains" for differential algebra over real, $p$-adic (and even characteristic $0$ pseudofinite) fields seems a worthwhile enterprise and this is what is accomplished in a fairly general context by Tressl's ``uniform model companion".  In the background may be also new kinds of semialgebraic $ODE$'s such as logarithmic differential equations on Nash groups, with their own Galois theory, but this remains to be seen.

There is also 	quite a bit of work on definable groups in various kinds of dense pairs, such as 
\cite{Baro-Martin-Pizarro}. As for example $RCF_{\partial}$ is an expansion of the theory of dense pairs of real closed fields, we expect both the results and methods of the current paper to impact this kind of work. 

Let us discuss briefly some connections of the work in this paper to other recent work, which were pointed out to us kindly by Pantelis Eleftheriou and Silvain Rideau-Kikuchi  after a first version of this paper was posted.   We focus on the $RCF_{\partial}$ case. Our first step, carried out in  Section 3, is to construct from a finite-dimensional definable group $\Gamma$ in $RCF_{\partial}$ a ``generically defined  group"  (see Definition 5.2) on a definable set in $RCF$.  See Proposition 3.5.  The second step is to produce from the generically defined group in $RCF$ a ``group chunk" on a definable set in $RCF$. See Lemma 4.2. The third step is to produce from the ``group chunk" an actual definable group $G$ in $RCF$. See Lemma 4.5. The fourth step is to definably (in $RCF_{\partial}$) embed $\Gamma$ in $G$.  See Lemma 4.7.  The second and third steps are put in the general context of geometric structures in our Section 5. 

In \cite{Pantelis}, a group chunk theorem appears in a certain axiomatic framework (of a finite-valued definable dimension on definable sets), yielding from a group chunk (in basically the same sense as ours) a definable group. This is Theorem 2.2 of \cite{Pantelis}.  The axiomatic framework includes ``geometric structures", so also $RCF$, $o$-minimal structures, pseudofinite fields, $p$-adically closed fields,...  Hence we could have appealed to this theorem for our third step, after checking and stating  various compatibilities, although the proof of our fourth step depends on the coherence of the constructions and notation from our steps one, two and three. 

Silvain Rideau-Kikuchi pointed out a couple of relevant papers \cite{Metastable} and \cite{Silvain-SCF}.  In Section 3.4 of \cite{Metastable}, a certain result is proved, Proposition 3.15, which after some degree of translation would, we understand,  recover a type-definable group from a ``generically given group" on a certain partial type $\pi$ where $\pi$ has some definability properties. 
This is closely related to the combination of our second and third steps (the content of our Section 5), although we produce a {\em definable} (rather than type-definable) group via the construction of the {\em definable} group chunk (in our sense).  

Although \cite{Silvain-SCF} is also related (see Proposition 4.1 there), it cannot be directly applied to our setting


In any case we  have decided to keep the material in the current paper as is. For example we are not able to directly quote any results from \cite{Metastable}, and a technical discussion about how those results may or may not  be modified suitably would take us rather outside our paper's scope.

\vspace{2mm}
\noindent
{\bf Acknowledgements.} In addition to the grant support mentioned above, all authors would like the thank the Fields Institute, Toronto, for their hospitality in Fall 2021 when this work began. In particular the second author would like to thank the Simons Foundation and the Fields Institute for a Simons Distinguished Visitor position during this period. The second author would also like to thank Imperial College for its hospitality in the summer of 2022 with a Nelder Visiting Fellowship.

\section{Preliminaries and geometric axioms}

In this section we work with the complete theory $RCF_{\partial}$

We take $L$ to be the the language of rings $\{+,\times, -, 0, 1\}$ and $L_{\partial}$ the language of differential rings $\{+,\times, -, 0, 1, \partial\}$. 

We will write a model $M$ of $RCF_{\partial}$ as $(K,\partial)$ where $K$ is a field, necessarily real closed, and $\partial$ is a derivation.
A saturated model of $RCF_{\partial}$ (or universal domain) will be written as $({\cal U}, \partial)$.  So note that ${\cal U}$ is a saturated real closed  field. 


For $A$ a subset of $\mathcal U$, by an $(L_{\partial})_{A}$-definable set we mean a set definable over $A$ in the structure $({\mathcal U}, \partial)$, so defined by an $L_{\partial}$-formula with parameters from $A$.  
By an $L_{A}$-definable set  over $A$ we mean the same thing but in the reduct of $\mathcal U$ to $L$, namely defined in the structure $\cal U$ by an $L$-formula over $A$.  So the latter is the same as a semialgebraic set over $A$. 
As usual we denote by $k({\bar a})$ the field generated by $k$ and ${\bar a}$  (where $k$ is a subfield of ${\cal U}$ and ${\bar a}$ a tuple from ${\cal U}$), and we let $k\langle{\bar a}\rangle$ denote the differential field generated by $k$ and ${\bar a}$ where ${\bar a}$ is as before and typically $k$ is a differential subfield of $({\cal U}, \partial)$.  Note that $k\langle{\bar a}\rangle
$ is precisely $k({\bar a}, \partial({\bar a}),..., \partial^{(m)}({\bar a}) ,....)$, where $\partial^{(m)} = \partial\circ ...
\circ\partial$ (m times). If ${\bar a} = (a_{1},...,a_{n})$, $\partial ^{(m)}({\bar a})$ denotes 
$(\partial^{(m)}(a_{1}),...,\partial^{(m)}(a_{n}))$. 
We also sometimes let $a^{(m)}$ denote $\partial^{(m)}(a)$, and likewise for  ${\bar a}^{(m)}$. 

\begin{Fact}
(i) Any $(L_{\partial})_{A}$-formula $\phi(\bar x)$ is equivalent, modulo $RCF_{\partial}$, to a formula of the form $\theta({\bar x},\partial({\bar x}),....,\partial^{m}({\bar x}))$ for some $m$ and $L_{A}$-formula  $\theta({\bar x}_{0},..,{\bar x}_{m})$. 
\newline
(ii) For any $A\subset {\mathcal U}$ the model theoretic algebraic closure of $A$ in the structure $({\mathcal U}, \partial)$ is the same as the field-theoretic relative algebraic closure of the differential subfield of $\mathcal U$ generated by $A$. Hence, by virtue of the ordering, the model-theoretic definable closure of $A$ in $({\mathcal U},\partial)$ is the same as the definable closure in the real closed field ${\mathcal U}$ of the differential field generated by $A$. 
\newline
(iii) If $X$, $Y$ are $(L_{\partial})_{A}$- definable sets in $({\mathcal U}, \partial)$ and $f:X\to Y$ is $(L_{\partial})_{A}$ definable then we can partition $X$ into finitely many $(L_{\partial})_{A}$-definable sets $X_{i}$ such that for each $i$, $f|X_{i}$ is of the form $g({\bar x},\partial({\bar x}),....,\partial ^{(m)}({\bar x}))$ for some $m$ and  $L_{A}$-definable function $g$. 
\end{Fact}
\noindent
{\em Explanation.} (i) is precisely quantifier elimination. (ii) follows from the Singer axioms for $CODF$, namely if $k$ is a (small) differential subfield of ${\mathcal U}$ and $a\in {\mathcal U}$ is not algebraic over $k$ in the field-theoretic sense then we conclude from the axioms and compactness/saturation that there are infinitely many elements in ${\mathcal U}$ with the same $L_{\partial}$ type as $a$ over $k$.
(iii) follows from quantifier elimination ,(ii),  and compactness.       \qed

\vspace{2mm}
\noindent
Typically $k, K, ...$ will denote small differential subfields of $({\mathcal U}, \partial)$. The following are exactly as in $DCF_{0}$:
\begin{Definition} 
(i) Let ${\bar a}$ be a finite tuple, and $k$ a differential subfield of $({\mathcal U}, \partial)$. By $order({\bar a}/k)$ we mean the transcendence degree over $k$ of $k\langle{\bar a}\rangle$  if this is finite, and $\infty$ otherwise. We will also write $dim_{\partial}({\bar a}/k)$ for $order({\bar a}/k)$ when the latter is finite. 
\newline
(ii) Let $X$ be definable over $k$. Then $order(X) = max\{order({\bar a}/k):{\bar a}\in X\}$ if this is finite and $\infty$ otherwise. We say that $X$ is ``finite-dimensional"  if $order(X)$ is finite. In this case we also write $dim_{\partial}(X)$ for $order(X)$. 
\newline
(iii) For $X$ definable over $k$ and finite-dimensional, and ${\bar a}\in X$ we will call $\bar a$ ${\partial}$-generic  in $X$ over $k$ if $dim_{\partial}({\bar a}/k) = dim_{\partial}(X)$. 
\newline
(iii) Let ${\bar a}$ and ${\bar b}$ be tuples (even infinite)  from ${\mathcal U}$. We will say that ${\bar a}$ is $\partial$-independent from ${\bar b}$ over $k$, if $k\langle{\bar a}\rangle$ is independent from  $k\langle{\bar b}\rangle$ over $k$ in the sense of fields.
\end{Definition}

We can and will distinguish $dim_{\partial}$, $\partial$-genericity, and $\partial$-independence from dimension, genericity, independence in $RCF$ by referring to the latter as $dim_{L}$, $L$-genericity, $L$-independence, when we want to avoid ambiguity. More will be said later, especially about the connection to the corresponding notions in algebraic geometry. When we want to distinguish types in the sense of $RCF$ and types in the sense of $RCF_{\partial}$, we will write $tp_{L}(...)$ and $tp_{L_{\partial}}(...)$ (or $tp_{\partial}(...)$ for brevity).


We will be talking about algebraic varieties $V$ over $\cal U$ or over (possibly differential) subfields $k$ of $\cal U$. 
First by a $k$-variety $V$ (or variety over $k$), we will usually mean an  {\bf affine} variety defined by polynomial equations over $k$ and which is $k$-irreducible. We will explain below why we will take in addition $V$ to be absolutely irreducible (also called geometrically irreducible).    But a warning is that we will attach to such $V$ another object $T_{\partial}(V)$ which is a possibly reducible affine variety over $k$. 
Secondly our notational conventions are that, in contrast to \cite{BCR} (where varieties are subsets of $K^{n}$, $K$ a real closed field), we identify a variety $V$ over $k$ with the ``functor" taking fields $F$ containing $k$ to the set $V(F)$ of solutions of the set of equations over $k$ defining $V$.
So from this point of view $V$ can be identified with its set of $F$ points for $F$ an algebraically closed field containing $k$.  On the other hand we will be concerned with sets of points $V(K)$ of $V$ in real closed fields $K< {\cal U}$ containing $k$ or in ${\cal U}$ itself. $V({\cal U})$ will of course be a set definable in ${\cal U}$ over $k$. 


There is no harm in assuming (as we will) that the base field $k<{\cal U}$ over which $V$ is defined is real closed.  The varieties we consider will typically arise by choosing a tuple ${\bar a}$ from ${\cal U}$ and defining $V$ to be the affine variety over $k$ defined by all the polynomial equations over $k$ vanishing at $\bar a$.  But then $V(k)$ will be Zariski dense in $V$ (as $k\prec {\cal U}$) and so in fact $V$ will be absolutely irreducible, not only $k$-irreducible. 
Let us mention in passing that from \cite{Pop}, the largeness of a field $k$ (such as a real closed field) is equivalent to any $k$-variety $V$ with a smooth $k$ point having a Zariski-dense set of $k$-points.



Given a $k$-variety $V$, 
the {\em algebro-geometric dimension} of $V$ is precisely the maximum of $trdeg(k(a)/k)$ where $a$ is a point of $V$ in some {\em algebraically closed} field $F$ containing $k$. We call such $a$ (witnessing the algebro-geometric dimension of $V$) a Zariski-generic  point of $V$ over $k$. 
To say that $V({\cal U})$ is Zariski-dense in $V$ means precisely that we can find a Zariski-generic point of $V$ over $k$ in $V({\cal U})$. So to summarize:

\begin{Remark} Assume $V$ is a $k$-variety such that $V({\cal U})$ is Zariski-dense. Then the algebro-geometric dimension of $V$ coincides with the $o$-minimal dimension of $V({\mathcal U})$. Hence for points $a\in V({\cal U})$ being Zariski-generic over $k$ and $o$-minimal generic  over $k$ coincide.
\end{Remark}


Now suppose $V$ to be a $k$-variety
From differential algebra we have the notion of the prolongation of $V$, also called the ``shifted tangent bundle of $V$", 
which will be a possibly reducible variety defined by polynomial equations over $k$. We denoted this by  $T_{\partial}(V)$, which comes equipped with a projection  $\pi:T_{\partial}(V) \to V$. $T_{\partial}(V)$ is also called $\tau(V)$ in the iterature (such as  \cite{Marker-Manin}). (So as mentioned earlier  there is a slight conflict with our assumption that all varieties we consider are irreducible.)

Here are the details. 
Assume that $V$ is a $k$-variety in affine $n$-space, and let $I_{k}(V)$ be the ideal in $k[x_{1},..,x_{n}]$ consisting of polynomials vanishing on $V$.  For each polynomial $P(x_{1},..,x_{n})\in I_{k}(V)$, consider the
equation over $k$ in variables $x_{1},..,x_{n}, u_{1},..,u_{n}$, 
$$\sum_{i=1,..,n}(\partial P/\partial x_{i})({\bar x})u_{i} + P^{\partial}({\bar x}) = 0$$  

Here $P^{\partial}$ is the polynomial over $k$ obtained from $P$ by applying the derivation $\partial$ to the coefficients of $P$.

Then $T_{\partial}(V)$ is the possibly reducible subvariety of affine $2n$-space defined by all these displayed equations, as well as $P({\bar x}) = 0$, as $P$ ranges over $I_{k}(V)$.  And $\pi$ is the projection to the first $n$-coordinates.

\begin{Definition} By a rational $D$-variety $(V,s)$ defined over a field $k$, we mean a $k$-variety $V$ equipped with a rational section, defined over $k$ from $V$ to $T_{\partial}(V)$. This means that $s$ is given by a system of 
quotients $P/Q$  of polynomial functions over $k$ in $n$ indeterminates, such that working in an ambient algebraically closed field, $s$ is defined on a nonempty Zariski open (over $k$) subset $U$of $V$ and for $a\in U$, $s(a)\in T_{\partial}(V)$ and $\pi(s(a)) = a$. 
\end{Definition}


When $s$ is a regular map, namely everywhere defined on $V$, we have the notion of an algebraic $D$-variety, as  defined explicitly by Buium \cite{Buium} in the general differential algebraic context, but with geometric origins in the notion of an Ehresmann connection. In the literature $V$ is often assumed to be smooth. So we state a fact, just for the record, although it will not be needed for the rest of the paper: 
\begin{Remark}  (i) When $V$ is smooth, then $T_{\partial}$ is irreducible (so also a $k$-variety) as well as smooth, with $dim(T_{\partial}(V)) = 2dim(V)$. 
\newline
(ii) In general, $T_{\partial}(V)$ may be reducible, with dimension $> 2dim(V)$. But nevertheless it is smooth and irreducible over the smooth locus $V_{sm}$ of $V$. Namely (again working in an ambient algebraically closed field), $\pi^{-1}(V_{sm})$ is a smooth irreducible quasi-affine variety over $k$. 
\newline
\end{Remark}
\begin{proof} 
(i) is (1.5) of \cite{Marker-Manin}.
\newline
(ii) The proof of (1.4) in \cite{Marker-Manin} gives smoothness of the quasi-affine variety  $\pi^{-1}(V_{sm})$, and the proof of (1.5) of \cite{Marker-Manin} gives irreducibility. 
\end{proof}

Now ${\mathcal U}$ embeds in a differentially closed field $({\bar{\cal U}},\partial)$
and the  geometric axioms for $DCF_{0}$ imply that one can find a generic point $a$ of $V({\bar{\cal U}})$ over $k$, such that $s(a) = (a,\partial(a))$. 

Lemma 1.6 of \cite{BCPP} says that the analogous statement holds for $RCF_{\partial}$:

\begin{Fact} Under the current assumptions  ($k$ a differential subfield of $\mathcal U \models RCF_{\partial}$, $(V,s)$ a rational  $D$-variety over $k$ and with $V({\mathcal U})$ Zariski-dense), there is 
$a\in V({\cal U})$, a generic point of $V$ over $k$ in the algebro-geometric sense, such that $s(a) = (a,\partial(a))$. 
\end{Fact} 

In fact, the conclusion characterizes (saturated) models of $RCF_{\partial}$.  Let us note that if ${\cal U}$ was replaced by a not necessarily saturated model $(K,\partial)$ of $RCF_{\partial}$ containing $(k,\partial)$, then the conclusion of Fact 2.6 would be that for every Zariski open subset $U$ of $V$, defined over $k$  there is $a\in U(K)$ with $s(a)$ defined and equal to $(a,\partial(a))$. 

The main thing we want to do now is to strengthen  Fact 2.6. We will need dimension in the $o$-minimal or semialgebraic sense as well as various compatibilities with dimension in the algebro-geometric sense (as in Remark 2.3).





We now return to our (saturated) model $({\mathcal U},\partial)$ of  $RCF_{\partial}$. We again let $k$ be (small) differential subfield of ${\cal U}$. Given a semialgebraic set $X$ defined over $k<\mathcal U$, let $a$ be an $L$-generic point of $X$ over $k$, and let $p(x) = tp_{L}(a/k)$. We call $p$ an $L$-generic type of $X$ over $k$.  Here is the strengthening of Fact 2.6, which can again be considered as giving {\em stronger} ``geometric axioms" for $RCF_{\partial}$.

\begin{Lemma} Let $k$ be a differential subfield of ${\mathcal U}$, Let $(V,s)$ be a rational $D$-variety over $k$ ($V$ affine, irreducible, $V({\mathcal U})$ Zariski-dense). Let $p(x)$ be an $L$-generic type of $V({\mathcal U})$ over $k$.  Then there is a realization $a$ of $p$ such that $s(a) = (a,\partial(a))$.
\end{Lemma}
\begin{proof}  First we may assume that $k$ is real closed. Let us fix a $|k|^{+}$-saturated model $K$ of $RCF$ extending $(k,+,\times,<)$.  So $p(x)$ is a complete type over $k$ in the sense of  $K$, so is realized by  some $a\in V(K)$.  By Remark 2.3, $a$ is Zariski-generic in $V$ over $k$.  Let $a = (a_{1},..,a_{n})$.  Let $s(a) = (a_{1},..,a_{n},b_{1},..,b_{n})$ so by definition we have that 
$\sum_{i=1,..,n}((\partial P/\partial x_{i})(a_{1},..,a_{n}))b_{i} + P^{\partial}(a_{1},..,a_{n}) = 0$ for all $P\in I_{k}(V)$.  By the extension theorem for derivations (Theorem 1.1 of \cite{Pierce-Pillay}) we can extend the derivation $\partial$ on $k$ to a derivation $\partial^{*}$ on the  field $k(a_{1},..,a_{n})$ by defining $\partial^{*}(a_{i}) = b_{i}$ for $i=1,..,n$.  Now extend $\partial^{*}$ to a derivation $\partial^{**}$ of $K$. So $(K, \partial^{**})$ is a model of ``$RCF$ + $\partial$ is a derivation", hence embeds in a model of $RCF_{\partial}$, so without loss of generality into $({\cal U}, \partial)$.  As $RCF_{\partial}$ has quantifier elimination in the ordered {\color{blue} differential} ring language (and $(k,\partial)$ is a differential field with $k$ real closed), $(K,\partial^{**})$ embeds in the 
ordered differential field ${\mathcal U}$ over $k$. The image of $a$ under this embedding works.
\end{proof}

\section{Definable groups of finite dimension}


We continue to use notation as at the beginning of Section 2. $({\cal U}, {\partial})$ still denotes a saturated model of $RCF_{\partial}$ in the language $L_{\partial}$. 


\begin{Definition} (i) If $a$ is a finite tuple from ${\mathcal U}$, and $N\geq 1$ then by $\nabla^{(N)}(a)$ we mean the
tuple  $(a,\partial(a),..,\partial^{(N)}(a))$. We write $\nabla(a)$ for  $\nabla^{(1)}(a) = (a,\partial(a))$.
\newline
(ii) For $X\subseteq {\cal U}^{n}$ a $L_{\partial}$-definable set, by $\nabla^{(N)}(X)$ we mean $\{\nabla^{(N)}(a):a\in X\}$. It is also a definable set, defined over the same parameters as $X$.
\end{Definition}

For $X$ as in Definition 3.1 (ii), the map $\nabla^{(N)}$ clearly gives a bijection between $X$ and $\nabla^{(N)}(X)$. When $X = \Gamma$ is a definable group (over a differential subfield $k$), then $\nabla^{(N)}$ induces a group structure on $\nabla^{(N)}(\Gamma)$ also defined over $k$. 

\begin{Fact}
Suppose $k$ is a (small) differential subfield of $({\cal U}, \partial)$, let ${\bar a}\in {\cal U}$ and suppose that 
$\partial^{(n)}({\bar a})$ is in the field-theoretic algebraic closure of 
\newline
$k({\bar a}, \partial({\bar a}),..,\partial^{(n-1)}
({\bar a}))$. Then $\partial^{(n+1)}({\bar a})\in k({\bar a},\partial({\bar a}),..,\partial^{(n-1)}({\bar a}), \partial^{(n)}({\bar a}))$.  
\newline
So also  $\partial^{(m)}({\bar a})\in k({\bar a},\partial({\bar a}),..,\partial^{(n-1)}({\bar a}), \partial^{(n)}({\bar a}))$  for all $m>  n$ (and trivially, for all $m\leq n$).  
\end{Fact}
\begin{proof} Let $a$ be a coordinate of the tuple ${\bar a}$ and we show that  $a^{(n+1)}\in k({\bar a},\partial({\bar a}),..,\partial^{(n-1)}({\bar a}), \partial^{(n)}({\bar a}))$. Let  $P(x)$ be the minimal (monic) polynomial of $a^{(n)}$  over $k({\bar a},..,\partial^{(n-1)}({\bar a}))$. 
Suppose $P$ has degree $m$.  We have that $P(a^{(n)}) = 0$.
So $\partial(P(a^{(n)})) = 0$.  We then compute that  $a^{(n+1)}(\partial P/\partial x)(a^{(n)}) \in k({\bar a},...,{\bar a}^{(n-1)}, {\bar a}^{(n)})$. But $\partial P/\partial x$ is a polynomial over $k({\bar a},..,{\bar a}^{(n-1)})$ of degree $< m$, whereby   
$(\partial P/\partial x)(a^{(n)}) \neq 0$ so we can divide by it to get $a^{(n+1)}\in k({\bar a},...,{\bar a}^{(n)})$ as required.

\end{proof}

The first aim is to find a reasonably canonical description of finite dimensional sets $X$ and groups $\Gamma$ definable in $({\cal U},\partial)$ over $k$, up to definable bijection (isomorphism) over $k$.  The definable bijection will simply replace $\Gamma$ by $\nabla^{(N)}(\Gamma)$ for suitable $N$.  This is an  adaptation of well-known  constructions in $DCF_{0}$, but with a few subtleties.


\begin{Proposition}\label{prop3.3}  (I) Let $X\subseteq {\cal U}^{n}$ be a finite-dimensional definable set in $({\cal U},\partial)$, defined over the small differential subfield $k$ of ${\cal U}$  (which may be assumed to be real closed). Then, after replacing $X$ by $\nabla^{(N)}(X)$ for suitable $N$  (and replacing $n$ by  $n(N+1)$),  there is a semialgebraic subset $Y$ of ${\cal U}^{n}$ defined over $k$ and irreducible affine subvarieties $V_{1},..,V_{r}$ of affine $n$-space, defined over $k$,  with Zariski-dense sets of $k$-points (or ${\cal U}$-points), and for each $i=1,..,r$, a rational section $s_{i}:V_{i}\to T_{\partial}(V_{i})$ defined over $k$, and Zariski open $U_{i}$ of $V_{i}$ defined over $k$ on which $s_{i}$ is defined, such that
\newline
(i) $Y$ is covered by the $U_{i}({\cal U})$,
\newline
(ii) For each $i$, $dim(V_{i}) = dim(Y \cap U_{i}({\cal U}))$, and 
\newline
(iii) $X = \cup_{i=1,..,r}\{a\in Y\cap U_{i}({\cal U}): \nabla(a) = s_{i}(a)\}$.

\vspace{2mm}
\noindent
(II) Suppose moreover that $X = \Gamma$ is a finite-dimensional definable group (defined again over $k$). Then, again replacing $\Gamma$ by the definably over $k$ isomorphic group  $\nabla^{(N)}(\Gamma)$ for suitable $N$, we have,  (with $\Gamma$ in place of $X$)
\newline
(iv) There is a partial semialgebraic (over $k$) function $*:Y\times Y \to Y$, whose restriction to $\Gamma\times\Gamma$ is precisely the group operation on $\Gamma$.
\newline
(v) There is a partial semialgebraic (over $k$) function $inv: Y\to Y$  whose restriction to $\Gamma$ is precisely group inversion. 

\end{Proposition}

\begin{proof}  
(I). By quantifier elimination, $X$ can be defined by  a formula 
\newline
$\theta({\bar x}, {\partial}({\bar x}), ..., {\partial^{(m)}}({\bar x}))$ where $\theta$ is an $L$-formula over $k$.  
Note that at this point we can see that  $\nabla^{(m)}(X)$ is defined as $$\{({\bar x}_{0}, {\bar x}_{1},...,{\bar x}_{m}): \theta({\bar x}_{0},....,{\bar x}_{m}), \partial({\bar x}_{0},...,{\bar x}_{m-1}) = ({\bar x}_{1},...,{\bar x}_{m})\},$$ which does not make use of the finite-dimensionality of $X$. We will, using finite-dimensionality, refine this description, so as to give the proposition. 

By finite-dimensionality and compactness we may find $N \geq m$ such that $\partial^{(N)}({\bar a})$ is in the field-theoretic algebraic 
closure of $k({\bar a},..., \partial^{(N-1)}({\bar a}))$ for all ${\bar a}\in X$.  By Fact 3.2, $\partial^{(r)}({\bar a})\in 
k({\bar a},..., \partial^{(N-1)}({\bar a}), \partial^{(N)}({\bar a}))$ for all ${\bar a}\in X$ and $r$ (in particular for $r = N+1$). 

Let us first fix  ${\bar a}\in X$, and let $V_{\bar a}$ be the affine variety over $k$ with generic (over $k$) point $\nabla^{(N)}({\bar a})$ ($ =  ({\bar a}, \partial ({\bar a}),...,\partial^{(N)}({\bar a}))$ remember).  We have just seen 
that  $\partial^{(N+1)}({\bar a})\in k(\nabla^{(N)}({\bar a}))$, so equals $f(\nabla^{(N)}({\bar a}))$ for some $k$-rational function $f$, defined at $\nabla^{(N)}({\bar a})$ so defined on some Zariski open set $U$ of $V_{\bar a}$ over $k$.  
We know that  $(\nabla^{(N)}({\bar a}), \partial(\nabla^{(N)}({\bar a})))$ is a point  on $T_{\partial}(V_{\bar a})$, so it follows that  $(\nabla^{(N)}({\bar a}), \partial(\nabla^{(N)}({\bar a}))) = s_{\bar a}(\nabla^{(N)}({\bar a}))$ for some rational over $k$ section $s_{\bar a}$ of $T_{\partial}(V_{\bar a}) \to V_{\bar a}$.

Note that this has been accomplished for every ${\bar a}$ in $X$, but the data in conclusions $V_{\bar a}$, $s_{\bar a}$  of course depend only on $tp_{\partial}({\bar a}/k)$. 

At this point it is convenient to {\em replace}  $X$ by $\nabla^{(N)}(X)$ via the $k$-definable bijection $\nabla^{(N)}$. 
So  the ``new" $X$ is  $$\{({\bar x}_{0},{\bar x}_{1},...,{\bar x}_{N}): \models \theta({\bar x}_{0},...,{\bar x}
_{m}), {\bar x}_{i+1} = \partial({\bar x}_{i}), i=0,..,N-1\}.$$ 

We will add dummy variables so as to consider the formula $\theta$ as having free variables ${\bar x}_{0},....,{\bar x}_{N}$.  We let $Y$ be the set of realizations of $\theta$, an $L_{k}$-definable set.

 We write a point of the ``new" $X$ as ${\bar b} = ({\bar b}_{0},.....,{\bar b}_{N})$.  So to summarize the situation with this new notation:
\newline
(*)  For each ${\bar b}\in X$ we have an irreducible affine variety $V_{\bar b}$ (in the appropriate affine space) over $k$  and rational (over $k$) section $s_{\bar b}: V_{\bar b} \to T_{\partial}(V_{\bar b})$ such that ${\bar b}$ is a ($o$-minimal or Zariski) generic point of $V_{\bar b}$ over $k$ and  $\nabla({\bar b}) = s_{\bar b}({\bar b})$.  Let $U_{\bar b}$ be a Zariski-open over $k$ subset of $V_{\bar b}$ on which $s$ is defined.

As $V_{\bar b}$, $U_{\bar b}$ and $s_{\bar b}$ depend only on $tp_{\partial}({\bar b}/k)$, we will write them as $V_{p}$, $U_{p}$ and $s_{p}$, where $p = tp_{\partial}({\bar b}/k)$.  Let $P$ be the collection of all such types. 

Note that:

\vspace{2mm}
\noindent
(**)  for any $p\in P$, if ${\bar b}\in Y \cap U_{p}$ and $\nabla({\bar b}) = s_{p}({\bar b})$ then ${\bar b}\in X$.

\vspace{2mm}
\noindent
This is because $\nabla({\bar b}) = s_{p}({\bar b})$ implies that $\partial({\bar b}_{i}) = {\bar b}_{i+1}$ for $i = 0,..,N-1$ which together with  ${\bar b}\in Y$ is equivalent to ${\bar b}\in X$. 

\vspace{2mm}
\noindent
On the other hand if ${\bar x}\in X$ then $ \bigvee_{p\in P}({\bar x}\in U_{p} \wedge 
\nabla({\bar x}) = s_{p}({\bar x}))$. 
So by compactness there are (distinct) $p_{1},..,p_{r}\in P$ such that $X$ is covered 
by $\cup_{i=1,..,r}\{{\bar b}\in U_{p_{i}}: \nabla({\bar b}) = s_{p_{i}}({\bar b})\}$. 

So together with (**) we obtain (i) and (iii) of (I) of the Proposition with $V_{i}, U_{i}$ for $V_{p_{i}}, U_{p_{i}}$ and $s_{i}$ for $s_{p_{i}}$. 
For (ii) notice that for each $i$, $Y\cap U_{i}$ is an $L_{k}$-definable subset of $V_{i}$ defined over $k$ which contains an $L$-generic over $k$  point of $V_{i}$ so has to have the same $L$-dimension as $V_{i}$.  So we have obtained (I).

\vspace{5mm}
\noindent
(II). Part (I) applies to $X = \Gamma$ as a definable set. The $k$-definable bijection $\nabla^{(N)}$ between $\Gamma$ and $\nabla^{(N)}(\Gamma)$ gives a $k$-definable group structure on $\nabla^{(N)}(\Gamma)$.  So we immediately assume that our definable group is $\nabla^{(N)}(\Gamma)$, which we still call $\Gamma$ and we have (I) (i), (ii), (iii). 
Applying Fact 2.1 (iii) to multiplication $m: \Gamma \times \Gamma \to \Gamma$ and using the fact that for $a\in \Gamma$, $k(a,\partial (a), ..., \partial ^{(m)}(a),.....) = k(a)$, there  is a partition of $\Gamma\times\Gamma$ into finitely many $L_{\partial}$-definable (over $k$) sets $Z_{i}$ such that on $Z_{i}$, $m$ is given by  an $L_{k}$-definable  function $h_{i}$ say.  Again $Z_{i}$ is the restriction to $\Gamma\times\Gamma$ of some $L_{k}$-definable  set $W_{i}\subseteq Y\times Y$, on which we may assume $h_{i}$ to be defined.  Let $W$ be the union of the $W_{i}$ (an $L_{k}$-definable set). Then by cut and paste we can  piece together coherently the $h_{i}$'s to find a $L_{k}$-definable function $h: W \to Y$ such that for each $i$ the restriction of $h$ to $Z_{i}$ is precisely $h_{i}$.  Hence the 
restriction of $h$ to $\Gamma\times \Gamma$ is multiplication.  Extend $h$ to a partial $L_{k}$-definable function 
$*: Y\times Y \to Y$, to give (iv).  Obtaining (v) is similar. 
\end{proof}  

From our definable (finite-dimensional) group $\Gamma$ we  will construct a semialgebraic analogue of Weil's pre-group \cite{Weil} (Section 1, including Proposition 1), which we could call a semialgebraic pre-group. 

We let $\Gamma$, $Y$, $V_{i}$, $U_{i}$, $s_{i}$, $*$, $inv$ be as in the conclusion of the proposition above.  
After reordering we may assume that $V_{1},..,V_{s}$ are of maximal dimension over $k$  (in the algebraic or semialgebraic sense) among the $V_{i}$.  So $dim_{L}(Y)$ coincides with the algebro-geometric dimension of $V_{i}$ ($= dim_{L}(V_{i}({\cal U}))$) for each $i=1,..,s$.  Note that any $L$-generic point of $Y$ over $k$ is in $U_{i}$. 

In the next lemma we give the compatibility of the various notions of dimension.  See Definition 2.2 and the comments following it. 

\begin{Lemma} 
Let $\Gamma$ be our finite-dimensional $L_{\partial}$-definable group (defined again over $k$). Then with the above notation (including having replaced $\Gamma$ by suitable $\nabla^{(N)}(\Gamma)$), we have:
\newline
(i)  $dim_{\partial}(\Gamma)$ is equal to $dim_{L}(Y)$.
\newline
(ii) Let $g\in \Gamma$. Then $g$ is ${\partial}$-generic in $\Gamma$ over $k$   iff $g$ is $L$-generic in $Y$ over $k$  iff for some $i=1,..,s$, $g\in V_{i}$ and is generic in $V_{i}$ over $k$ in the algebro-geometric sense. 
\newline
(iii) Let $g_{1},..,g_{k}\in \Gamma$. Then the $g_{i}$ are $\partial$-independent over $k$ iff they are 
$L$-independent over $k$. 
\newline
(iv)   For any $d$, let $x_{1},...,x_{d}$ be $L$-generic $L$-independent (over $k$) elements of $Y$. Then there are $g_{1},..,g_{d}\in \Gamma$ which are $\partial$-generic, $\partial$-independent over $k$, such that $tp_{L}(x_{1},..,x_{d}/k) = tp_{L}(g_{1},.,g_{d}/k)$.
\newline
(v) If $g_{1}$, $g_{2}$ are $\partial$-generic in $\Gamma$ over $k$ and $\partial$-independent over $k$, then $g_{1}g_{2}$ (product in $\Gamma$) is also $\partial$-generic in $\Gamma$ over $k$ and $\{g_{1}, g_{2}, g_{1}g_{2}\}$ is pairwise $\partial$-independent over $k$.
\end{Lemma}
\begin{proof}  (i) We first note that for all $g\in \Gamma$, $k\langle g \rangle$ = $k(g)$, whereby $dim_{\partial}(g/k) = order(g/k) = trdeg(k(g)/k)$ which $= dim_{L}(g/k)$.  Hence $dim_{\partial}(\Gamma) \leq dim(Y)$.
\newline
On the other hand, fix $i=1,..,s$.  By 
Fact 2.6 there is $a\in V_{i}$ which is $L$-generic over $k$  and with $\nabla(a) = s_{i}(a)$. But then $a\in \Gamma$. So $dim_{\partial}(\Gamma)\geq dim(V_{i}) = dim_{L}(Y)$.
\newline
(ii) has the same proof.
\newline
(iii)  is because $k\langle g_{i} \rangle = k(g_{i})$. 
\newline
(iv) For each $i=1,..,d$ let $j_{i}\in \{1,..,s\}$ be such that  $x_{i}\in V_{j_{i}}$.  Then $(x_{1},..,x_{d})$ is $L$-generic in $V_{j_{1}}\times ..\times V_{j_{d}}$ over $k$. But $(V_{j_{1}}\times ..\times V_{j_{d}}, s_{j_{1}} \times ...\times s_{j_{d}})$ is  a rational $D$-variety over $k$ (satisfying the additional conditions in the hypotheses of Lemma 2.7, namely irreducibility with a Zariski-dense set of ${\cal U}$-points). 
By Lemma 2.7, there is $(g_{1},..,g_{d})$ satisfying $tp_{L}(x_{1},..,x_{d}/k)$ such that $\nabla (g_{i}) = s_{j_{i}}(g_{i})$ for each $i$. But then, because of the construction
each $g_{i}\in \Gamma$, and clearly the $g_{i}$ are $\partial$-generic and $\partial$-independent, over $k$. 
\newline
(v)  follows by properties of $L$-dimension 
\end{proof}


Using the notation of Proposition \ref{prop3.3}, we have
\begin{Proposition} (i) For $L$-generic $L$ independent (over $k$) $x,y$ in $Y$, $x*y$ is defined and is also 
$L$-generic in $Y$ over $k$.
\newline
(ii) For $x,y,z$ $L$-generic $L$-independent over $k$, $(x*y)$ is $L$-independent of $z$ over $k$ and $x$ is $L$-independent of $(y*z)$ over $k$ and $((x*y)*z)= (x*(y*z))$. 
\newline
(iii) For $L$-generic (over $k$) $x\in Y$, $inv(x)$ is defined and $L$-generic in $Y$ over $k$, and $inv(inv(x)) = x$. Moreover if $y$ is also $L$-generic and $L$-independent of $x$ over $k$, we have that  $y = inv(x)*(x*y)$, $x = (x*y)*inv(y)$. 
\end{Proposition}
\begin{proof} 
(i)  By Lemma 3.4 above, let $g,h\in \Gamma$ such that  $tp_{L}(g,h/k) = tp_{L}(x,y/k)$. 
Then $tp_{L}(g,h,gh/k) = tp_{L}(x,y,x*y/k)$.  By Lemma 3.4 ((ii) and (v)), $gh$ is $L$-generic in $Y$ over $k$, so $x*y$ is too. 
\newline
(ii) Again by Lemma 3.4 (iv) we may assume that $x,y,z$ belong to $\Gamma$, and by Lemma 3.4 the statement reduces to its truth for $\partial$-independence,
which is given by Lemma 3.4(v). 
\newline
(iii)  Same proof as above. 

\end{proof}

\section{Producing an $RCF$-definable group. }

The main in this section  is to use Proposition 3.5 and an adaptation of various group construction results to prove:

\begin{Theorem} Let $\Gamma$ be a finite-dimensional group definable in $({\mathcal U}, \partial)$. Then there is a group $G$ definable in the real closed field ${\mathcal U}$ and an $L_{\partial}$-definable embedding of $\Gamma$ into $G$.  
\end{Theorem}

See also Remark 4.8 at the end of this section for more on the relation between $\Gamma$ and $G$.

Starting from the data  $Y$, $*$, $inv$ in Proposition 3.5 the construction of $G$ is related to Weil's theorem 
about recovering a (connected)  algebraic group from birational  data, namely from a generically given group operation on an irreducible variety $V$ \cite{Weil}.  In the Weil context, the generically defined operation is on a single stationary  type in $ACF_{0}$ (the unique generic type of $V$). In our context there is an unbounded set of $L$-generic types of our semialgebraic set $Y$, which introduces some complications. Nevertheless we produce a suitable ``group-chunk" $Y_{0}\subseteq Y$ in Lemma 4.2 below, from which we will construct our semialgebraic group from a  collection of germs of definable, invertible functions on $Y_{0}$.  Similar things appear in a stability-theoretic context in Hrushovski's work on invertible germs of definable functions on a stationary type \cite{Hrushovski-locallymodular}. The point of view taken here is also influenced by van den Dries' paper \cite{vdDries}.  We put this aspect of our work into the general context of geometric structures, in Section 5.1. 

There is another aspect to Weil's theorem which concerns constructing not just a definable (in $ACF$) group, but an {\em algebraic} group.  The $RCF$-analogue of this aspect is about constructing a {\em Nash group} which is done in  \cite{Pillay-groupsandfields}. This will be relevant to the follow-up paper on Nash $D$-groups.

We will be working for now in the structure ${\mathcal U}$  (i.e. a saturated real closed field). In any case for now definable (and generic, independent etc.)  means in ${\mathcal U}$, namely semialgebraic or $L$-definable. At the end of the section (Lemma 4.7) we will come back to definability in $({\mathcal U},\partial)$. 

We will go straight in to the proof of Theorem 4.1.   We use notation as in Proposition 3.5, namely ($Y$, $*$,  $inv$).  By a {\em large} subset $Y_{0}$ of $Y$ defined over $k$ we mean a semialgebraic subset $Y_{0}$ of $Y$ with $dim(Y\setminus Y_{0}) <  dim(Y)$.  In particular $dim(Y_{0}) = dim(Y)$.  This is equivalent to saying that all generic over $k$ points of $Y$ are in $Y_{0}$.  Note that the intersection of finitely many large subsets of $Y$ is large. 

Note that if a formula $\phi(x)$ over $k$ holds  for all $k$-generic $a\in Y$, then $\phi(x)$ defines a large subset of $Y$. Likewise for subsets of $Y\times Y$, $Y\times Y \times Y$.
We will also use freely definability of dimension in $RCF$.

We now let $F$ denote the partial map $*: Y\times Y \to Y$, and $g$ the partial map $inv$.


\begin{Lemma} Let  $Y_{0}$ be the set of $x\in Y$ such that 
\newline
(i) for a large set of  $(y,z)$ in $Y^{2}$, $F(x,y)$, $F(y,z)$, $F(x,F(y,z))$, $F(F(x,y),z)$ are defined and $F(x,F(y,z))=F(F(x,y),z)$,
\newline 
(ii) $g(x)$ and $g(g(x))$ are defined and $g(g(x)) = x$, 
\newline
(iii)  for a large set of $y\in Y$,  $g(y)$, $F(x,y)$, and $F(g(x), F(x,y))$ are defined, $F(g(x), F(x,y)) = y$,
$F(F(x,y), g(y))$ is defined and equals $x$,  and  $g(F(x,y)) = F(g(y), g(x))$.
\newline
Then $Y_{0}$ is definable over $k$ and is a large subset of $Y$. 
\end{Lemma}
\begin{proof} This is immediate, using Proposition 3.5,  once one sees that $Y_{0}$ is definable over $k$ and contains all generic over $k$ points of $Y$. 
\end{proof}

We now consider germs of partial definable functions from $Y_{0}$ to $Y_{0}$.  Note first that if $f(-)$ is a partial  function definable with parameters (such as $F(x,c)$ for some $c$),  then the following are equivalent  (see \cite{Pillay-groupsandfields})
\par (i) for some 
tuple of parameters $d$ such that $f$ is definable over $k,d$, whenever $y\in Y_{0}$ is generic over $k,d$ then $f(y)$ is defined and in $Y_{0}$, 
\par (ii) there is a large definable (with parameters) subset $U$ of $Y_{0}$ such that $f$ is defined on $U$, with values in $Y_{0}$. 

If $f_{1}$, $f_{2}$ satisfy (i) (and (ii)) above, we will say that $f_{1}$ and $f_{2}$ have the same {\em germ} on $Y_{0}$, if whenever $y\in Y_{0}$ is generic over the relevant parameters, then $f_{1}(y) = f_{2}(y)$, equivalently, by the above, $f_{1}=  f_{2}$ on some large definable subset of $Y_{0}$. 

We let $\sim$ be the equivalence relation of having the same germ. Then $\sim$ is definable over $k$ for uniformly definable over $k$ families of functions. Namely if $f(x,z)$ is a partial function definable over $k$ such that for all $c$, $f(-,c)$ satisfies (i) (and (ii)) above then the equivalence relation $E(z_{1}, z_{2})$: $f(-,z_{1})\sim f(-,z_{2})$ is $k$-definable, so the collection of germs of the $f(-,c)$ as $c$ varies is a $k$-definable set in ${\cal U}^{eq}$  (which is the same as ${\cal U}$, as $RCF$ eliminates imaginaries).


We will consider (partial) definable functions of the form $F_{a}(-)  = F(a,-)$ for $a\in Y_{0}$.  We show that they are ``generically invertible", the set of germs of definable functions $Y_{0} \to Y_{0}$ generated by the set of $F_{a}$ is a group which is generated in two steps  (and even better).  This will be the sought after definable group.


We will work {\em over $k$}. 

\begin{Lemma} (i) $g$ is a bijection from $Y_{0}$ to itself, and $g\circ g$ is the identity on $Y_{0}$.{\color{blue}$\cdot$}
\newline
(ii) For every $a\in Y_{0}$, there is a large subset $U_{a}$ of $Y_{0}$ such that $F_{a}$ is a bijection between $U_{a}$ and its image $F(U_{a})$ in $Y_{0}$ which is also large in $V_{0}$. So we can compose any $F_{a}$ and $F_{b}$ to get another function with the same properties.
\newline
(iii) For every $a\in Y_{0}$, $a$ is determined by the germ of $F_{a}$.
\newline
(iv)  For every generic $a\in Y_{0}$ and any $x,y\in Y_{0}$, if $F_{a}\circ F_{x}$ and $F_{a}\circ F_{y}$ have the same germ then $x= y$. 
\newline
(v)  If  $x\in Y_{0}$ and $a$ is generic over $x$, then $F_{x}\circ  F_{a}$ has the same germ as $F_{b}$ for some (generic) $b\in dcl(x,a)\cap Y_{0}$.  Likewise  $F_{a}\circ F_{x}$ has the same germ as $F_{b}$ for some (generic) $b\in dcl(x,a)\cap Y_{0}$.
\newline
(vi)  The set of germs of $F_{a}\circ F_{b}$, for $a,b\in Y_{0}$, is a group, under composition of germs, where moreover the inverse of the germ of $F_{a}\circ F_{b}$ is the germ of $F_{g(b)}\circ F_{g(a)}$. 
\end{Lemma}
\begin{proof} (i) is immediate from Lemma 4.2 (ii). 
\newline
(ii) Fix $a\in Y_{0}$. Let $U_{a}$ be $\{y\in Y_{0}: F_{a}(y)$ is defined and in $Y_{0}$, and $F_{g(a)}(F_{a}(y))$ is defined and equals $y$\}. Then $U_{a}$ is definable (over $a$). Suppose $y\in Y_{0}$ is generic over $a$, so by 4.2 (ii),  $F_{a}(y)$ is defined and $F_{g(a)}(F_{a}(y)) = y$ and in particular $F_{a}(y)$ is also generic over $a$ so in $Y_{0}$. Hence if $y\in Y_{0}$ is generic over $a$, then $y\in U_{a}$. It follows that $U_{a}$ is large in $Y_{0}$.

Note that it follows that $F_{a}|U_{a}$ is a bijection between $U_{a}$ and its image, which by definition is also a subset of $Y_{0}$.

It remains to show that the image $F_{a}(U_{a})$ is large. But as in the first part of the proof of (ii) (and the fact that $g(a)\in Y_{0}$) $U_{g(a)}$ is large. But $U_{g(a)}$ is contained in $F_{a}(U_{a})$ so the latter is large too.

So as stated in the lemma, given $a, b\in Y_{0}$, the composition  $F_{a}\circ F_{b}$ will give a bijection between two large subsets of $Y_{0}$.
\newline
(iii)  Let $a,b\in Y_{0}$ and suppose that $F_{a} \sim F_{b}$. Choose $c\in Y_{0}$ generic over $a,b$.  Then $F_{a}(c) = F_{b}(c) = d$ say.  
By Lemma 4.2 (iii), $F_{d}(g(c))$  is defined and equals $a$, and also equals $b$, so $a=b$. 
\newline
(iv) Suppose that $F_{a}\circ F_{x} \sim F_{a}\circ F_{y}$. Let $z$ (in $Y_{0}$ of course) be generic over $a,x,y$. So $F_{a}(F(x,z)) = F_{a}(F(y,z))$.   By part (ii) $F(x,z) = F(y,z)$. By Lemma 4.2 (iii), $F(F(x,z),g(z)) = x$ and $F(F(y,z), g(z)) = y$, so $x=y$. 
\newline
(v)  First  consider  $F_{x}\circ F_{a}$.  Let $y$ be generic over $x,a$, then $(F_{x}\circ F_{a})(y) = F(x, F(a,y)) = F(F(x,a),y)$ 
by Lemma 4.2(i) and properties of dimension. But  $F(x,a) = b$ is generic (over $x$ too) and clearly $y$ is generic over $x,a,b$ so $F_{x}\circ F_{a}$ has the same germ as $F_{b}$. 
\newline
Now consider $F_{a}\circ F_{x}$.  By the first part of our proof of (v), $F_{g(x)}\circ F_{g(a)}$ has the same germ as $F_{c}$ for some generic $c$. It suffices to prove:
\newline
{\em Claim.}  $F_{a}\circ F_{x}$ has the same germ as  $F_{g(c)}$. 
Choose $y$ generic over $x,a,c$. Let $F(x,y) = w$ and $F(a,w) = z$. So $(F_{a}\circ F_{x})(y) = z$. Then  $w,z$ are each generic over $x,a,c$, and   $F_{g(a)}(z) = w$ and $F_{g(x)}(w) = y$, whereby  $(F_{g(a)}\circ F_{g(x)})(z) = y$, hence $F_{c}(z) = y$, whereby $F_{g(c)}(y) = z$, proving the Claim.
\newline
(vi)  The fact that the inverse of the germ of $F_{a}\circ F_{b}$ is the germ of $F_{g(b)}\circ F_{g(a)}$ is proved as in 
the last part of the proof of (v), namely that for generic $y$ over $a,b$, and $z = (F_{a}\circ F_{b})(y)$, $z$ is also 
generic over $a,b$ and moreover $(F_{g(b)}\circ F_{g(a)})(z) = y$. 

In order to prove that $\{F_a\circ F_b:a,b
\in Y_0\}$ is a group it suffices to prove that for $a,b,c,d\in Y_{0}$,  $F_{a}\circ F_{b}\circ F_{c}\circ F_{d}$ has the same germ as  $F_{x}\circ F_{y}$ for some $x,y\in Y_{0}$. 

Let $e\in Y_{0}$ be generic over $a,b,c,d$, and consider $F_{a}\circ F_{b}\circ F_{e}\circ F_{g(e)}\circ F_{c}\circ F_{d}$, which has the same germ as $F_{a}\circ F_{b}\circ F_{c}\circ F_{d}$.

Iterating (v) we see that $F_{a}\circ F_{b}\circ F_{d}$ has the same germ as $F_{x}$ (where  $x$ is generic), and likewise  $F_{g(e)}\circ F_{c}\circ F_{d}$ has the same germ as some $F_{y}$.
So putting it together, $F_{a}\circ F_{b}\circ F_{c}\circ F_{d}$ has the same germ as $F_{x}\circ F_{y}$. 

\end{proof}

From (vi) above we have constructed a group  interpretable in $RCF$ (over $k$)  whose domain is the set of germs of $F_{a}\circ F_{b}$, $a,b\in Y_{0}$. The domain is the quotient of $Y_{0}^{2}$ by a $k$-definable equivalence relation $E$.
We call this group $G$. 

As $RCF$ has elimination of imaginaries, this group is definably isomorphic over $k$ to one whose underlying set is a definable subset of ${\mathcal U}^{m}$ some $m$. 
However, it will be convenient, also for other contexts where one does not have elimination of imaginaries,  to have a finer statement which is what we do now.

\begin{Lemma} Let $N =  dim(Y_{0})$, and let $\alpha > 2N$.  Let $t_{1},...,t_{\alpha}$ be independent generic elements of $Y_{0}$. Then for any $a,b\in Y_{0}$, there is $d\in Y_{0}$ and $t\in \{t_{1},..,t_{\alpha}\}$ such that $F_{a}\circ F_{b}$ has the same germ as $F_{t}\circ F_{d}$. Moreover $d$ is determined by $t$ and the germ of $F_{t}\circ F_{d}$. 
\end{Lemma}
\begin{proof}   Given $a,b$ we can find, by dimension considerations, some $t\in \{t_{1},..,t_{\alpha}\}$ such that $t$ is generic over $a,b$.  Consider  $F_{t}\circ F_{g(t)}\circ F_{a}\circ F_{b}$ which has the same germ as $F_{a}\circ F_{b}$.
But as in the proof of (vi) of the previous Lemma, $F_{g(t)}\circ F_{a}\circ F_{b}$ has the same germ as $F_{d}$ for some (generic) $d\in Y_{0}$. Hence $F_{a}\circ F_{b}$ has the same germ as $F_{t}\circ F_{d}$. 

The last part of the statement is Lemma 4.3 (iv). 
\end{proof}

We can now write  $G$ naturally as a definable (rather than interpretable) group, as follows. 

By Lemma 4.4 we can think of $G$ as covered by the charts $\{t_{i}\}\times Y_{0}$ (with $(t_{i},a)$ sent to the germ of $F_{t_{i}}\circ F_{a}$). In order to write this as a disjoint union, we remove from $\{t_2\}\times Y_0$ all germs which already appear in $\{t_1\}\times Y_0$, and then remove from $\{t_3\}\times Y_0$ all germs which already appear previously and so forth. This gives us our group living naturally on a collection of tuples of a given length.




On the face of it, the {\em definable} group $G$  is defined over $k$ together with parameters $t_{1},..,t_{\alpha}$. However as the real closure of $k$ is an elementary substructure of $\mathcal U$ and is also equal to the definable closure of $k$ in ${\mathcal U}$ we can choose $t_{1},..,t_{\alpha}$ in $dcl_{L}(k)$.\\

 

The end result is:
\begin{Lemma} The group of germs of $F_{a}\circ F_{b}$ for $a,b\in Y_{0}$ is an interpretable group $G$ which is the same as the set of germs of $F_{t_{i}}\circ F_{b}$, $i=1,..,\alpha$, $b\in Y_{0}$ which can be identified naturally as above with a definable (over $k$)  group.
\end{Lemma}

\begin{Remark}
We know (\cite{Pillay-groupsandfields}) that a group definable in $RCF$ has ` definably the structure of a Nash group. In a subsequent paper, we will use this to get a Nash group structure on $G$, compatible in a sense with the way in which  $G$ originates from the $L_{\partial}$-definable group $\Gamma$.
\end{Remark}

We now return to consideration of the group $\Gamma$, $L_{\partial}$-definable in $({\cal U},\partial)$. So we revert to our notation distinguishing the different notions of generic etc. 

The proof of Theorem 4.1 is completed by the following, where we again work over $k$, so suppress $k$ from the 
notation. 
\begin{Lemma} There is an $L_{\partial}$-definable group embedding  $h$ of $\Gamma$ into $G$.
\end{Lemma}
\begin{proof} Let $\gamma \in \Gamma$.  Choose $a\in {\Gamma}$ $\partial$-generic over $\gamma$, and let 
$b = a^{-1}\gamma$ in $\Gamma$.  So $\gamma = ab$ in $\Gamma$.  Note that $b$ is also $\partial$-generic in $\Gamma$. Hence by Lemma 3.4 both $a$ and $b$ are $L$-generic in $Y$, so in  $Y_{0}$.  Define $h(\gamma)$ to be the germ of  $F_{a}\circ F_{b}\in G$.
We have to check several  things:
\newline
(i) $h$ is well-defined, i.e. $h(\gamma)$ does not depend on the choice of $a$ above,
\newline
(ii) $h$ is $\partial$-definable. 
\newline
(iii) $h$ is a group homomorphism, and
\newline
(iv) $h$ is injective. 
\newline
{\em Proof of (i).} Suppose $\gamma$ also equals $cd$ where $c,d$ are $\partial$-generic in  $\Gamma$, so $L$-generic in $Y_{0}$. 
We want to show that $F_{a}\circ F_{b}$ and $F_{c}\circ F_{d}$ have the same germ. Choose $x\in Y_{0}$ $L$-generic over $a,b,c,d$. By Lemma 2.7 we may assume that $x\in \Gamma$. Hence working in $\Gamma$, $abx = cdx$. But for generic independent (in either sense, by Lemma 3.4) elements of $\Gamma$, multiplication in $\Gamma$ is given by $*$, by Proposition 3.3 (iv).  As $b,x$ are independent generic elements of $\Gamma$, as well as $a, bx, c, dx$, it follows that $a*(b*x) = c*(d*x)$. But this means precisely that $F_{a}\circ F_{b}$ and $F_{c}\circ F_{d}$ agree at $x$. We have shown that $F_{a}\circ F_{b} \sim F_{c}\circ F_{d}$ as required. \\
(We could alternatively have made use of the fact (Proposition 3.3) that $F(a,b)$ coincides with the group operation in $\Gamma$, together with associativity (Lemma 4.2).)

\vspace{2mm}
\noindent
{\em Proof of (ii).}  It is enough, by compactness,  to show that the graph of $h$ is $\partial$-type-definable.
By the definition of $h$ and part (i), we have, for $\gamma \in \Gamma$ and $g\in G$, $h(\gamma) = g$ iff and only if for some $a\in \Gamma$ which is $\partial$-generic and $\partial$-independent from $\gamma$, $g$ is the germ of $F_{a}\circ F_{a^{-1}\gamma}$.  This suffices. 

\vspace{2mm}
\noindent
{\em Proof of (iii).} Suppose that $\gamma = ab$ and $\delta = cd$ with $a,b,c,d$ $\partial$-generic in $\Gamma$. 
So $h(\gamma)$ is the germ of $F_{a}\circ F_{b}$ and $h(\delta)$ is the germ of $F_{c}\circ F_{d}$.  We have to show that $h(\gamma\delta)$ is the germ of $F_{a}\circ F_{b}\circ F_{c}\circ F_{d}$.
As in the proof of Lemma 4.3 (v),  let $f\in \Gamma$ be $\partial$-generic and $\partial$-independent of $a,b,c,d$.  Then $\gamma\delta = (abf) (f^{-1}cd)$  (multiplication in $\Gamma$).  Let $x = abf$ and $y = f^{-1}cd$. Then $x,y$ are $\partial$-generic in $\Gamma$ and $\gamma\delta = xy$. Finally one has to check that that the germ of $F_{x}$ is  equal to the germ of $F_{a}\circ F_{b}\circ F_{f}$ and the germ of $F_{y}$ equals the germ of $F_{f^{-1}}\circ F_{c}\circ F_{d}$, which is like in the proof of (i).
So $h$ is a homomorphism.

\vspace{2mm}
\noindent
{\em Proof of (iv).} Given (i), (ii), (iii), we have to prove that if $h(\gamma)$ is the identity then $\gamma$ is the identity.  Suppose that $\gamma = ab$ as before and $F_{a}\circ F_{b}$ is the identity.  So $F_{a}\circ F_{b}$ has the same germ as $F_{a}\circ F_{a^{-1}}$. By Lemma 4.3 (iv), $b = a^{-1}$, so $\gamma = aa^{-1}$ is the identity of $\Gamma$.

\end{proof}

Let us record the following additional information which will be useful in future work. 
\begin{Remark} The embedding $h$ of $\Gamma$ in $G$ from Lemma 4.7 also satisfies:
\newline
(i) any $L$-generic type of $G$ over $k$ is realized by some element of (the image under $h$ of) $\Gamma$.
\newline
(ii) There is a covering of $G$ by finitely many $L_{k}$-definable subsets $X_{i}$, and there are suitable $L_{k}$ definable functions $s_{i}$ on $X_{i}$, such that for any $i$, if $a\in X_{i}$ is $L$-generic in $G$ over $k$ then $a$ is in (the image under $h$ of) $\Gamma$ iff $\partial(a) = s_{i}(a)$. 
\end{Remark}
\begin{proof} Everything can be seen by inspecting the proofs above. 
However some precisions are carried out in the more general Theorem 5.4 in the next section.
Anyway, with notation from the current (and previous)  section, from 4.3 (iii)  we have an injective map $h_{1}$ from $Y_{0}$ to $G$ taking $a\in Y_{0}$ to the germ of $F_{a}$.  We leave the reader to check that for $a\in \Gamma\cap Y_{0}$, $h_{1}(a) = h(a)$. 
By Theorem 5.4 (1) (and translating) $h_{1}(Y_{0})$ is large in $G$, so  $a\in Y_{0}$ is $L$-generic over $k$ iff $h_{1}(a)\in G$ is $L$-generic over $k$. 
Also note that $Y_{0}$ is large in $Y$ so contains all generics of $Y$. 
Now Lemma 3.4 gives a covering of $Y_{0}$ by the $V_{i}$, and for $a\in V_{i}\cap Y_{0}$, $a\in \Gamma$ iff $\nabla(a) = s_{i}(a)$. Moreover any $L$-generic over $k$ element of $Y_{0}$ has $L$-type realized by a ($\partial$-generic over $k$) element of $\Gamma$.
Applying the injective map $h_{1}$ to $Y_{0}$ gives (i) and (ii) of the Remark.

\end{proof}

\section{Geometric structures and fields}

In this section we present some straightforward generalizations of the methods and results from the earlier sections.

\subsection{Geometric structures and generically defined groups}
We are here interested in generalizing that aspect of Section 4 which produces a semialgebraic group from the data  (a semialgebraic pre-group) in Proposition 3.5.  There are various kinds of generalizations, including just axiomatizing the proof. But a fairly clearly stated and useful generalization takes place in the context of  geometric structures from \cite{Hrushovski-Pillay-groupslocalfields}. We recall the definition and give a quick summary of properties from Section 2 of \cite{Hrushovski-Pillay-groupslocalfields}. 

\begin{Definition} A geometric structure  is a one sorted structure $M$ in a language $L$  in which algebraic closure yields a pregeometry in any model of $Th(M)$ and which has elimination of the $\exists^{\infty}$  quantifier, namely for each formula $\phi(x,
{\bar y})$ there is $N_{\phi}$ such that for any ${\bar b}$ in $M$, $\phi(x,{\bar b})(M)$ is infinite if it has cardinality $>N_{\phi}$.  
\end{Definition}

Geometric structures include strongly minimal structures, o-minimal; structures, structures of $SU$-rank $1$,  as well as $p$-adically closed fields.  Moreover in the case of real closed fields, or $p$-adically closed fields,  model-theoretic algebraic closure coincides with (relative)  field-theoretic algebraic closure.

\vspace{5mm}
\noindent
Let us now fix a saturated geometric structure $M$, where as usual $A,B$ denote small subsets and $a,b,,..$ will now denote finite tuples of elements of $M$.

We define $dim(a/A)$ to be the (unique) cardinality of a maximal algebraically independent over $A$ subtuple of $a$. 
Then $a$ and $b$ are said to be independent over $A$ if $dim(a,b/A) = dim(a/A) + dim(b/A)$. For $X\subseteq M^{n}$ definable (or even type-definable) over $A$, $dim(X) = max\{dim(a/A): a\in X\}$. $a\in X$ is said to be {\em generic over $A$} in $X$ if $dim(X) = dim(a/A)$.  Note that for $p$ a complete $n$ type over $A$, $dim(p) = dim(a/A)$ for some/any realization $a$ of $p$. For tuples $a,b$, we have
$$\dim(a,b/A)=\dim(a/bA)+\dim(b/A).$$

Elimination of the $\exists^{\infty}$ quantifier implies that for any $L$-formula $\phi(x,y)$ and $m$, $\{b: dim(\phi(x,b)) = m\}$ is definable over $\emptyset$ in $M$.

In the case of  $M$ being a real closed field, $dim$ and independence coincide with what was described in Section 2.

\begin{Definition} By a {\em generically defined group} (or a pre-group in $M$) we mean a definable set $X\subseteq M$  (defined over $A$ say) together with $A$-definable  partial operations $F: X\times X \to X$, and $g:X\to X$ such that:
\newline
(1) For any  $a,b\in X$, generic and independent over $A$, $c = F(a,b)$ is defined, and each of $a,b,c$ is in the definable closure of $A$ and the other two, 
\newline
(2)  For $a,b,c \in X$, generic and independent over $A$, $F(F(a,b),c) = F(a,F(b,c))$. 
\newline
(3) For $a\in X$ generic over $A$, $g(a)$ is defined and generic in $X$ over $A$.  Moreover, if $b$ is also generic in $X$ over $A$ and independent of $a$ over $A$, we have $b = F(g(a),F(a,b))$, and $a = F(F(a,b), g(b))$. 
\end{Definition}


\begin{Remark}
(i) We could replace (3) in the definition above, by
\newline 
(3)': there are $A$-definable partial functions $\mu$, $\lambda$ from $X\times X$ to $X$, which satisfy (1), and if $a,b,c\in X$ are generic and pairwise independent over $A$, then $F(a,b) = c$ iff $\mu(a,c) = b$ iff $\lambda(b,c) = a$. 
\newline
(ii) When $X$ is an an irreducible algebraic variety, and the ambient structure $M$  is an algebraically closed field (of characteristic $0$ say), then parts (1) and (2) of the definition above are precisely the definition of a ``pre-group equipped with a normal law of composition $f$" as in 
 \cite{Weil}. 
\end{Remark}

\noindent
{\em Explanation.} (i) (3)' is easily a consequence of (1), (2), (3). And (3) follows from (1), (2), (3)' by copying the proof of Proposition 1 in Weil's \cite{Weil}.
\newline
(i) In this context of $ACF_{0}$
for each $n$ there is a unique $n$-type of $n$ generic, independent elements of $X$. Then one deduces (3)' as in the proof of Proposition 1 of \cite{Weil}.       \qed

\begin{Theorem}  Let $(X,F,g)$ be a generically defined group in $M$, as in Definition 5.2. Then there is a group $(G,\cdot)$ definable in $M$ over a set $B\supseteq A$, 
with $dim(G) = dim(X)$, and a partial $B$-definable function $h:X\to G$ such that:
\newline
(1) $h$ establishes a bijection between generic over $B$ points of $X$ and generic over $B$ points of $G$, so in particular $h(X)$ is large in $G$, 
\newline
(2) For $a,b\in X$ generic and independent over $B$, $h(F(a,b)) = h(a)\cdot h(b)$.
\newline
(3) For $a\in X$ generic over $B$, $h(g(a)) = h(a)^{-1}$.
\newline
Moreover when $A$ is an elementary substructure of $M$ we can choose $B = A$. 
\end{Theorem} 
\begin{proof}
This is given by the proofs in Section 4 which work under the more general assumptions.  
Here are some details.  After changing some symbols, the definition of $(X,F,g)$ being a generically defined group (over $A$), gives the statements in Proposition 3.5 above.

Our notation is now almost the same as in Section 4, except we have $X$ in place of $Y$ and $A$ in place of the field $k$ (over which everything is defined).  By the definability properties of geometric structures, Lemma 4.2 goes through giving the suitable large subset $X_{0}$ of $X$ defined over $A$  (the ``group chunk").  

Now we consider definable (with additional parameters) functions from $X_{0}$ to $X_{0}$ and their germs, where recall that $f_{1}$ and $f_{2}$ have the same germ if  whenenever $a\in X_{0}$ is generic over $A$ together with the parameters in $f_{1}$ and $f_{2}$, $f_{1}(a) = f_{2}(a)$. Equivalently there is a large definable subset of $X_{0}$ on which $f_{1}$ and $f_{2}$ agree. 

Then Lemma 4.3 goes through, giving the interpretable group $G_{0}$ as the group of germs of $F_{a}\circ F_{b}$ for $a,b\in X_{0}$.  Now Lemmas 4.4 and 4.5 go through to find elements $t_{1},..,t_{\alpha}$ of $X_{0}$ such that taking $B = A\cup\{t_{1},..,t_{\alpha}\}$, $G_{0}$ is definably over $B$ isomorphic with a {\em definable} (over $B$) group $G$.  As in the comments before Lemma 4.5, when $A$ is an elementary substructure of $M$, the $t_{i}$ can be found in $A$ already.  
As $G$ is covered by finitely many copies of $X_{0}$, and $X_{0}$ is large in $X$ we already see that $dim(G) = dim(X)$.  

We now choose $h$. There  is no harm assuming $A = B$. 
Let $h$ be the $A$-definable function on $X_{0}$ taking $a\in X_{0}$ to the germ of $F_{a}$.  So by (the general version of) Lemma 4.3(iii), $h$ is an injective $A$-definable function from $X_{0}$ to $G$. (So also $h$ is a partial $A$-definable function from $X$ to $G$ defined on the generic points of $X$.)  As $dim(X) = dim(G)$, we see immediately that  for $a\in X_{0}$ generic over $A$, $h(a)$ is generic in $G$ over $A$. 


Conversely, to  finish the proof of (1) we want to check that every generic over $A$ element of $G$ is the germ of $F_{a}$ for some generic over $A$ element $a$ of $X_{0}$.  From what we have just seen, $dim(X) = dim(G)$. Let $g$ be generic over $A$ in $G$. So we can write $g$ as the germ of $F_{u}\circ F_{v}$ for $u,v\in X_{0}$.  Let $b\in X$ be generic over $A, u, v$. Then $g$ times the germ of $F_{b}$ is the germ of $F_{a}$ for some generic $a\in X_{0}$ over $A$ by applying Lemma 4.3(v) twice. But then a dimension count shows that $a$ and $b$ are independent over $A$, as are $a$ and $g(b)$.  Then the germ of $F_{a}$ times the germ of $F_{g(b)}$ equals $g$, but also equals the germ of $F_{c}$ for $c$ generic in $X$ over $A$ by Lemma 4.3 (v). We have shown that $g$ is the germ of $F_{c}$ for $c$ generic in $X$ over $A$.  This completes the proof of (1).  

For (2), if $a,b$ are generic and independent over $A$ and are in $X_{0}$ then $F(a,b)$ is generic over $A$  so in $X_{0}$. $h(a)$ is the germ of $F_{a}$, $h(b)$ is the germ of $F_{b}$, and, since $F(a,F(b,c))=F(F(a,b),c)$, then the germ of $F_{a}\circ F_{b}$ is the germ of $F(a,b)$, giving (2). 

 (3) is immediate.
\end{proof}

\subsection{Geometric fields with a generic derivation}
We will now place the main results for finite-dimensional groups definable in $RCF_{\partial}$ in a certain general context which subsumes other cases of interest.

We are concerned first with certain theories of fields in the ring language $L$ possibly with additional constant symbols.  Then we will add a ``generic derivation" and work in $L_{\partial}$.

We will repeat some definitions from \cite{Hrushovski-Pillay-groupslocalfields} and \cite{BCPP}.
The notion of a ``geometric field" was introduced in \cite{Hrushovski-Pillay-groupslocalfields} in terms of ``geometric substructures" of algebraically closed fields. But in Remark 2.10 of \cite{Hrushovski-Pillay-groupslocalfields} the following characterization was given which we may take as a definition.

\begin{Definition} A geometric field is a perfect field $F$, considered as an $L$-structure, where $L$ is the language of rings with possibly additional constant symbols, such that in any model of $Th(F)$, model-theoretic algebraic closure coincides with field-theoretic (relative) algebraic closure, and for any $L$-formula $\phi(x,{\bar y})$ ($x$ a single variable) there is $N_{\phi}$ such that in any model $F'$ of $Th(F)$ and for any ${\bar b}$ in $F'$, $\phi(x,{\bar b})$ has finitely many solutions in $F'$ iff it has at most $N_{\phi}$ solutions.
\end{Definition} 

Note that a geometric field $K$ is also a geometric structure.  Moreover if $\bar a$ is an $n$-tuple from $K$ and $k$ a subfield, relatively algebraically closed in $K$ if one wishes, then $dim({\bar a}/k) = dim(V)$ where $V$ is the variety over $k$ with generic over $k$ point ${\bar a}$.


In a recent preprint, Johnson and Ye showed that if $M$ is a pre-geometric field, namely the model-theoretic algebraic closure has  exchange, then $M$ is geometric, namely it eliminates $\exists^{\infty}$ (and it is perfect) \cite[Theorem 5]{JY}.

We will restrict our attention to characteristic $0$. We recall the strengthening of model-completeness from \cite{BCPP}:
\begin{Definition} (i) Let $T$ be a theory of fields in the language $L$  of rings. 
We say that $T$ has almost quantifier elimination (almost QE) if whenever $F$ is a model of $T$, $A$ is a relatively algebraically closed subset of $F$ (in the field-theoretic sense) and ${\bar a}$ is an enumeration of $A$,  then $T\cup qftp({\bar a})$  implies a complete type $p({\bar x})$ (over $\emptyset$) of $T$.  
\newline
(ii) 
 Let $T'$ be a theory of differential fields in the language $L_{\partial}$. We say that $T'$ has almost QE, if whenever $(K,\partial)$ is a model of $T'$, and $k$ is a differential subfield of $K$ which is relatively algebraically closed in $K$ in the sense of fields, and ${\bar a}$ is an enumeration of $k$ then $T'\cup\{qftp({\bar a})\}$ implies a complete $L_{\partial}$-type (over $\emptyset$). 
\end{Definition}

Note that model completeness of $T$  is the special case of Definition 5.6(i) when $A = F$. 

We have already defined large fields. Here is a summary of relevant results from \cite{BCPP} and \cite{Tressl}.
\begin{Fact} Let $T$ be a model complete theory of large fields in the language $K$. Then:
\newline
(i) $T \cup \{\partial$ is a derivation\} has a model companion $T_{\partial}$ in the language $L_{\partial}$.
\newline
(ii) (The geometric axioms.) $(K,\partial)\models T_{\partial}$ iff $K\models T$, and whenever $(V,s)$ is an irreducible  rational $D$-variety over $K$ with a smooth $K$-point, and $U$ is a Zariski open subset of $V$ defined over $K$ then there is $a\in U(K)$ with $s(a) = \nabla(a)$.
\newline
(iii) If $T$ has quantifier elimination in a definitional expansion $L^{*}$, then $T_{\partial}$ has quantifier elimination in $L^{*}_{\partial}$.
\newline
(iv)  If $T$ has almost QE, then so does $T_{\partial}$. 
\end{Fact}
\noindent
{\em Explanation.} Everything can be extracted from Tressl's paper \cite{Tressl}, 7.1 and 7.2. But (ii) and (iv) are also proved in \cite{BCPP}. (See also Remark 2.4 in \cite{BCPP}.)

\vspace{5mm}
\noindent
We now generalize Lemma 2.7, even though in the proof of 2.7 we made use of $RCF_{\partial}$ having QE in the ordered differential field language. 

\begin{Lemma} Let $T$ be a model-complete theory of large geometric fields. 
Let $T_{\partial}$ be the model companion of $T\cup\{\partial$ is a derivation\}. 
Let $(K,\partial)\models T_{\partial}$ and $(V,s)$ be a (irreducible) rational $D$ variety over $K$ such that $V$ has a Zariski-dense set of $K$-points. Let $X$ be a definable (in $K$) subset of $V(K)$ with $dim_{L}(X) = dim(V) (=dim_{L}(V(K)))$. Then there is some $a\in X$ with $s(a) = \nabla(a)$.
\end{Lemma}
\begin{proof} The proof is like that of Lemma 2.7, with the proviso mentioned earlier.   Let $F$ be a saturated elementary extension of the field $K$ in the language $L$. Let $p(x)$ be a complete type over $K$ containing the formula $x\in X$, and with $dim(p) = dim(X)$. 
Let $a$ realize $p(x)$ in $F$. So $a$ is a generic point of $V$ over $K$. Let $s(a) = (a,b)\in F$, and as in the proof of Lemma 2.7 we can extend the derivation $\partial$ to a derivation $\partial^{*}$ of $F$ such that $\partial^{*}(a) = b$. 
So $(F,\partial^{*})$ is a model of $T$ together with $\partial^{*}$ is a derivation, so embeds in a model of $T_{\partial}$ which by model-completeness is an elementary extension of $(K,\partial)$. This suffices.
\end{proof}

\begin{Lemma} \label{acl}Let $T$ be a model complete theory of large geometric fields. Then in any model $(K,\partial)$ of $T_{\partial}$, the model theoretic algebraic closure of a set $A$ equals the field-theoretic algebraic closure of the differential subfield of $K$ generated by $A$.
\end{Lemma}
\begin{proof} It is enough to prove the following: Suppose $(K,\partial)$ is model of $T_{\partial}$, $k$ is a differential subfield of $K$, $K$ is $|k|^{+}$-saturated and $a\in K$ is not in the field-theoretic algebraic closure of $K$. Then $tp(a/k)$ in $(K,\partial)$ has infinitely many realizations in $K$.  This is a routine consequence of the previous lemma and the fact that  $tp_{L_{\partial}}(a/k)$ is determined by $tp_{L}(a,\partial(a), \partial^{(2)}(a),...)/k)$ (by 5.7 (iii)).
\\


\end{proof}

Now we bring in the almost QE hypothesis.
\begin{Proposition}  Assume that $T$ is a theory of large geometric fields which has almost QE.
Let $(K,\partial)\models T_{\partial}$ and let $k$ be a differential subfield of $K$. Then $dcl_{L_{\partial}}(k)$
is the same as $dcl_{L}(k)$. 
\end{Proposition}
\begin{proof} Let $a\in dcl_{L_{\partial}}(k)$. By the previous lemma we know that $a$ is in the field-theoretic algebraic closure of $k$.  Hence it suffices to prove:
\newline
{\em Claim.} Suppose $a$ is in the field-theoretic algebraic closure of $k$, then $tp_{L}(a/k)$ implies $tp_{L_{\partial}}(a/k)$. 
\newline
{\em Proof of Claim.}  Suppose $tp_{L}(b/k) = tp_{L}(a/k)$.  So assuming (as we may) that $K$ is $|k|^{+}$-saturated, there is a field automorphism $\sigma$ of $K$ fixing $k$ and taking $a$ to $b$. Let $k'$ be the algebraic closure of $k$ in $K$ (which by Lemma \ref{acl}, equals the relative field-theoretic algebraic closure of $k$ in $K$).  So $\sigma$ restricts to a permutation of $k'$ over $k$ taking $a$ to $b$. Thus $\sigma|k'$ preserves quantifier-free $L$-types over $k$. But $\partial$ has a unique extension to a derivation of $k'$, whereby $\sigma|k'$ also preserves quantifier-free types over $k$ in $L_{\partial}$. By Fact 5.7(iv), $a$ and $b$ have the same type over $k$ in $(K,\partial)$, proving the Claim and the Proposition. 

\end{proof} 

 Fix a model $(K,\partial)$ of $T_{\partial}$, where $T$ is a theory of large geometric fields with almost QE.  There is no harm assuming $(K,\partial)$ to be saturated. All of Section 2 goes through for $T$ in place of $RCF$, using 5.7, 5.8, 5.9, and 5.10. In particular, using  5.7(iii), any  $L_{\partial}$-formula $\phi({\bar x})$ is equivalent, modulo $T_{\partial}$ to one of the form $\theta({\bar x}, \partial({\bar x}),.....)$ where $\theta$ is an an $L$-formula.
Finite-dimensional sets definable in $(K,\partial)$ are defined as in Definition 2.2.  The various symbols for dimension, generics, independence, in $T$ and $T_{\partial}$ are as in Section 2.


Proposition 5.10 allows the proofs in Section 4 to go through, yielding:
\begin{Theorem} Let $(K,\partial)$ be a model of $T_{\partial}$ where $T$ is a not necessarily complete
theory of large geometric fields with almost $QE$. 
Then any finite-dimensional group $\Gamma$ definable in $(K,\partial)$ definably embeds in a group definable in the field $K$ (i.e. in the ring language $L$). Moreover Remark 4.8 also holds. 
\end{Theorem}
\begin{proof} We will again give just a brief guide to the proof, as the proof of Theorem 4.1 just goes through, given the comments in the previous two paragraphs.    Let $\Gamma$ be a finite-dimensional group definable in $(K,\partial)$ over the small differential subfield $k$, which may be assumed to be an elementary substructure. 


First Proposition 3.3 goes through for $\Gamma$: after replacing $\Gamma$ by $\nabla^{(N)}(\Gamma)$ for suitable $N$, we obtain a set $Y$ $L$-definable over $k$, satisfying (I), and (II) there.  For (II) Proposition 5.10 is needed. 
Then Lemma 3.4 goes through with the same proof, using Lemma 5.8 to see that any $L$-generic type of $Y$ can be realized by a realization of some $\partial$-generic type of $\Gamma$.  Proposition 3.5 follows as before. 
Now repeat Section 4 to get an $L$-definable over $k$ group $G$ into which $\Gamma$-definably embeds. Alternatively, $(Y,*,inv)$ is a generically defined group over $k$ in the $L$-structure $K$ in the sense of Definition 5.2, so Theorem 5.4 applies to produce the $L_{k}$-definable  group $G$.  The proof of 4.7 gives a $L_{\partial}$-definable over $k$ embedding of $\Gamma$ in $G$.

Remark 4.8 goes through too by the same proof as at the end of Section 4. 

\end{proof}


Some examples of theories satisfying the assumptions of the last proposition are (the theories of) bounded pseudoalgebraically closed fields of characteristic $0$ (including pseudofinite fields of characteristic $0$) \cite{Hrushovski-PAC}, and (the theory of) ${\mathbb Q}_{p}$. 

\section{Other variants}

We consider two other general contexts which go outside that of Section 5.2.

\subsection{The Fornasiero-Kaplan setting: $o$-minimal structures with a generic derivation}
Here we discuss $o$-minimal expansions of real closed fields with a generic derivation, as appearing in 
\cite{F-K}.  

Formally $T$ is taken to be a model-complete $o$-minimal expansion of $RCF$. $T\cup\{\partial$ is a ``compatible" derivation\} is shown to have a model companion $T_{\partial}$. Compatibility will be defined below, but when the language is just that of rings, any derivation is compatible. 
So for  $T = RCF$, we obtain $RCF_{\partial}$.
We take now $L$ to be the language of $T$ (which is in general an expansion of the language of rings) and $L_{\partial}$ to be $L\cup\{\partial\}$. 

We will again prove that finite-dimensional groups definable in models $(K,\partial)$ of $T_{\partial}$ are definably embeddable in groups definable in the language $L$. The proof will be like the earlier ones except for the first steps where cell-decomposition in models of $T$ replaces algebraic varieties and their ``prolongations". More accurately the only algebraic varieties that need to be considered are the affine spaces $K^{n}$.  These $o$-minimal methods could also have been used in the $RCF_{\partial}$ case but we chose instead to use techniques which could be generalized as we did in Section 5.2. 

We will give a brief account of the theory from \cite{F-K}, but with  notation consistent with the  current paper. Then we sketch again the proof of the desired result. 

For concrete applications $T$ can be an arbitrary model-complete (and complete) $o$-minimal theory in a language $L$ expanding $RCF$. We now assume that a symbol for the ordering $<$ is in the language $L$. For the kind of general results we have in mind there is no harm in assuming that $T$ has quantifier elimination (and is universally axiomatized). 
For models $M$ of $T$ we will freely use basic notions from $o$-minimality. ``$L$-algebraic" now means in the sense of the $o$-minimal structure $M$. 

The paper \cite{F-K} defines a derivation $\partial$ of the underlying field $M$ to be {\em compatible} with the $o$-minimal structure on $M$ if for every $n$, open definable subset $U$ of $M^{n}$, and $C^{1}$-function $f:U\to M$ which is $\emptyset$-definable in $M$, we have that for all ${\bar a}\in U$, 
$\partial (f({\bar a})) = \sum_{i=1,..,n}(\partial f/\partial x_{i})({\bar a})\partial(a_{i})$.
We assume now that $\partial$ is compatible with $M$ (so with any model of $T$). 

Note that compatibility  follows from the definition of a derivation in the case that $f$ is a polynomial over the constants of $M$.  We will discuss later the case when $f$ is a  $C^{1}$- definable function with parameters. 

 From Section 4  of \cite{F-K} we have:
\begin{Fact}
(1) $T\cup\{\partial$ is a compatible derivation\} has a model companion $T_{\partial}$ which is complete, and moreover has quantifier elimination in $L_{\partial}$ assuming $T$ does.
\newline
(2) $T_{\partial}$ can be axiomatized (modulo $T\cup\{\partial$ is a compatible derivation\}) by the following:
Suppose $X\subseteq M^{2n}$ is $L$-definable (with parameters) and its projection to $M^{n}$ has $dim = n$ (i.e. contains an open set) then there is ${\bar a}\in M^{n}$ such that $({\bar a}, \partial({\bar a})) \in X$. 
\newline
(3) For any $L_{\partial}$-formula $\phi({\bar x})$ there is an $L$-formula $\theta({\bar x}_{0},...,{\bar x}_{N})$ 
(with $length({\bar x}) = length({\bar x}_{0}))$ such that for any $(M,\partial)\models T_{\partial}$, and ${\bar a}$ 
from $M$, we have $(M,\partial)\models \phi({\bar a}) \leftrightarrow  \theta({\bar a},...,\partial^{(N)}({\bar a}))$. 

\end{Fact}

Given $(M,\partial)\models T_{\partial}$, we let $dcl_{L_{\partial}}$ denote definable closure in the  sense of the structure ($M,\partial)$ and $dcl_{L}$ for definable closure in the sense of the $o$-minimal structure $M\models T$.

From Corollary 5.14 of \cite{F-K} we have:
\begin{Fact} Let $A\subset M$ where $(M,\partial)\models T_{\partial}$. Then $dcl_{L_{\partial}}(A) = 
dcl_{L}(\{\partial^{(n)}(a): a\in A: n=0,1,2, ...\})$. In other words $dcl_{L_{\partial}}(A)$ is the definable closure in the $L$-structure $M$ of the differential field generated by $A$.
\end{Fact}

We now work in a saturated model $(M,\partial)$ of $T_{\partial}$. $k$ typically denotes a (small) differential subfield of $(M,\partial)$.  We will use $dim_{L}$,  $L$-generic etc. to mean in the sense of $o$-minimal structure $M$.

As in the earlier sections we will use  the notation $\nabla^{(N)}({\bar a})$ for $({\bar a}, \partial {\bar a}, ...,\partial^{(N)}({\bar a}))$.  Similarly for $X$ a definable set in $(M,\partial)$, $\nabla^{(N)}(X) = \{\nabla^{(N)}({\bar a}): {\bar a}\in X\}$

\begin{Definition} By $dim_{\partial}({\bar a}/k)$ we mean $\infty$ or $max\{dim(\nabla^{(N)}({\bar a})/k): N=0,1,....\}$ if there is a maximum. 
\end{Definition}

Note that  $dim_{\partial}({\bar a}/k)$ is an invariant of $dcl_{L_{\partial}}(k({\bar a}))$ (and $k$). So if ${\bar a}$ and ${\bar b}$ are $L_{\partial}$-interdefinable over $k$ then $dim_{\partial}({\bar a}/k) = dim_{\partial}({\bar b}/k)$. 

As in the earlier sections we will define an $(L_{\partial})_{k}$ definable  set to be {\em finite-dimensional} if there is a finite bound on $dim_{\partial}({\bar a}/k)$ for ${\bar a}\in X$ and will define $dim_{\partial}(X)$ to be this finite bound.

\begin{Definition} For $X$ an $(L_{\partial})_{k}$-definable set of  finite-dimension,  and ${\bar a}\in X$ we say that ${\bar a}$ is $\partial$-generic in $X$ over $k$, or $tp_{L_{\partial}}({\bar a}/k)$ is a $\partial$-generic type of $X$, if $dim_{\partial}({\bar a}/k) = dim_{\partial}(X)$. 
\end{Definition}

We have to discuss  the application of $\partial$ to elements of the form $f({\bar a})$ where $f$ is an $L$- definable {\em with parameters} function from some open subset $U$ of $M^{n}$ to $M$.
First we can write $f({\bar x})$ as $g({\bar x}, {\bar b})$ where $g({\bar x}, {\bar y})$ is some $L$-definable over $\emptyset$ function from an open subset of $M^{n+m}$ to $M$ and ${\bar b} = (b_{1},..,b_{m})$ consists of independent nonalgebraic over $\emptyset$ elements of $M$.   With this notation and using the compatibility of $\partial$, we have:
\begin{Remark}  (With the conventions above.) For ${\bar a}\in U$, 

$$\partial (f({\bar a})) =  \sum_{i=1,..,n}(\partial f/\partial x_{i})({\bar a}){\partial}(a_{i}) + f^{\partial}({\bar a}),$$ where $f^{\partial}({\bar a})$ is  $\sum_{j=1,..,m} (\partial g/\partial y_{j})({\bar a}, {\bar b}){\partial}(b_{j})$. 
\end{Remark}

\begin{Corollary} Suppose $\partial^{(n+1)}({\bar a})\in dcl_{L}(k(\nabla^{(n)}({\bar a})))$. 
Then, for all $m$,  $\partial^{(m)}({\bar a})\in dcl_{L}(k(\nabla^{(n)}({\bar a})))$.
\end{Corollary}

Finally we need the analogue of Lemma 2.7:
\begin{Lemma} Let $p({\bar x})$ be an $L$-generic type of $M^{n}$ over $k$ (which just means that $p = tp_{L}({\bar a}/k)$ where $dim_{L}({\bar a}/k)  = n$), and let $s({\bar x})$ be a partial function from $K^{n}$ to $K^{n}$ definable in $L$ over $k$ and defined at ${\bar a}$. Then there is a realization ${\bar b}$ of $p$ such that $\partial({\bar b}) = s({\bar b})$.
\end{Lemma}
\begin{proof} Let $U\subseteq K^{n}$ be open and $L$-definable over $k$ such that ${\bar a}\in U$ and $s$ is defined and $C^{1}$ on $U$. By Fact 6.1 (2), for every formula $\phi({\bar x})\in p({\bar x})$ there is a realization ${\bar b}$ of $\phi$ such that $\partial({\bar b}) = s({\bar b})$. Apply compactness (in the saturated model $(M,\partial)$). 
\end{proof}

We also have of course $o$-minimal independence in $M$ and  corresponding notion of $\partial$-independence in 
$(M,{\partial})$ where  ${\bar a}$ and ${\bar b}$ are $\partial$-independent over $A$ if the $L_{\partial}$-definable closures of ${\bar a}$ and ${\bar b}$ are $L$-independent over the $L_{\partial}$-definable closure of $A$. We will use freely the obvious properties of $\partial$-independence, as in the $RCF_{\partial}$ case. 

We can now give the proof of the following, where items (ii) and (iii) are useful for further work. 
\begin{Theorem} (Assume $T$ is a model-complete, complete, $o$-minimal expansion of $RCF$ in language $L$, and let $(M,\partial)$ be a saturated model of $T_{\partial}$.
Suppose that $\Gamma$ is a finite-dimensional group definable over $k$ in $(M,\partial)$. Assume $k$ is a (small) elementary substructure of $(M,\partial)$. Then 
there is a group $G$, $L$-definable over $k$ in $M$ such that
\newline
(i) there is an $L_{\partial}$-definable (over $k$) embedding $h$ of $\Gamma$ into $G$.
\newline
(ii) Every $L$-generic over $k$ type of $G$ is realized by an element of $h(\Gamma)$. 
\newline
(iii) There are $L$-definable over $k$ sets $X_{1},..,X_{r}\subseteq G$  $L$-definable functions $s_{i}: X_{i}\to M^{n}$ for suitable $n$, such that  $\cup_{i}X_{i} = G$, and for each $L$-generic over $k$ element $a\in G$, if $a\in X_{i}$ then $a\in \Gamma$ iff $s_{i}(a) = \partial(a)$. 
\end{Theorem}
\begin{proof} As remarked above, the only change from the proof of Theorem 4.1 concerns some of the preparatory material in Section 3. This approach will also work for $RCF_{\partial}$ which will be a useful observation for later work. 

Fix the group $\Gamma$, which by Fact 6.1(3) is defined by a $(L_{\partial})_{k}$-formula of the form $\theta({\bar x}, \partial({\bar x}),...,\partial^{(N)}({\bar x}))$ for some $L$-formula $\theta$ over $k$ and some $N$.
By finite-dimensionality of $\Gamma$, we may assume that  for all ${\bar a}$,  $\theta({\bar a}, \partial({\bar a}),...,\partial^{(N)}({\bar a}))$ implies  $\partial^{(N)}({\bar a})\in dcl_{L}(k((\nabla^{(N-1)}({\bar a})))$.
By Corollary 6.6 for each ${\bar a}\in \Gamma$, $\partial^{(N+1)}({\bar a})\in dcl_{L}(\nabla^{(N)}({\bar a}))$, so 
equals $f(\nabla^{(N)}({\bar a}))$ for some $L_{k}$-definable  partial function $f$ defined at $\nabla^{(N)}({\bar 
a})$ so defined on the realizations of some $L_{k}$-formula $\psi_{\bar a}(x_{0},...,x_{N})$ in $tp_{L}(\nabla^{N}({\bar a})/k)$. 
Replace $\Gamma$ by  $\nabla^{(N)}\Gamma$.  So the new $\Gamma$ is defined by  $\theta(x_{0},..,x_{N})$ together with $\partial(x_{i}) = x_{i+1}$ for $i=0,..,N-1$.
The new ${\bar x}$ is $({\bar x}_{0},..,{\bar x}_{N})$.
\newline
We let $Y$ be the $L$-definable over $k$ set defined by $\theta({\bar x})$.
\newline
 $\bar a$ will now denote a tuple from  $\Gamma$, so of the form 
${\bar a} = ({\bar a}_{0},...,{\bar a}_{N})$.  And for ${\bar a}$ in  $\Gamma$, $\partial({\bar a}) = ({\bar a}_{1},..,
{\bar a}_{N}, \partial({\bar a}_{N}))$.  Moreover by the previous paragraph, for each such ${\bar a}\in \Gamma$, $
\partial({\bar a}) = s_{{\bar a}}({\bar a})$ for some $L$-definable over $k$ function $s_{\bar a}$ which is defined on 
some formula $\phi_{\bar a}({\bar x})\in tp_{L}({\bar a}/k)$.  We may assume $dim_{L}(\phi_{\bar a}({\bar x}))$ equals $dim_{L}({\bar a}/k)$. 

 Note that for ${\bar a}\in \Gamma$, $\phi_{\bar a}$ and $s_{\bar a}$ depend only on $p = tp_{L_{\partial}}({\bar a}/k)$ so can be written as $\phi_{p}$, $s_{p}$. We let $P$ be the collection of $tp_{L_{\partial}}({\bar a}/k)$, for ${\bar a}\in \Gamma$.

As in the proof of  the first part of 3.3, we have:
\newline
{\em Claim 1.}  For $p\in P$, and ${\bar b}$ satisfying $\theta({\bar x})\wedge \phi_{p}({\bar x})\wedge ``\partial({\bar x}) = s_{p}({\bar x})"$, we have that ${\bar b}\in \Gamma$. 

\vspace{2mm}
\noindent
By compactness, we conclude,
\newline
{\em Claim 2.} There are $p_{1},..,p_{s}\in P$, such that $$\Gamma = \cup_{i=1,..,s}\{{\bar b}: \models \theta({\bar b})\wedge \phi_{p_{i}}({\bar b}) \wedge ``\partial({\bar b}) = s_{p_{i}}({\bar b})"\}.$$ 

\vspace{2mm}
\noindent
This already gives a kind of analogue of part (I) of Proposition 3.3. 
The analogue of part (II) follows by the same proof.  
\newline
{\em Claim 3.} There are partial $L$-definable over $k$ functions $*:Y\times Y \to Y$, and $inv:Y\to Y$ such that $*$ is defined on $\Gamma\times \Gamma$ and is precisely multiplication in $\Gamma$, and such that $inv$ is defined on $\Gamma$ and is precisely inversion on $\Gamma$. 


Note that:
\newline
{\em Claim 4}. For ${\bar a}\in \Gamma$, $dim_{L}({\bar a}/k) = dim_{\partial}({\bar a}/k)$. 

\vspace{2mm}
\noindent
We now will make use of cell decomposition, focusing on the $V_{p_{i}}$ (and $p_{i}$) of maximal $\partial$-dimension.  Let $n= dim_{\partial}(\Gamma) = dim_{L}(Y)$.

Take a $C^{1}$-cell decomposition of $Y$ by $C^{1}$-cells defined over $k$ which is compatible with the collection of formulas $\phi_{p_{i}}({\bar x})$ and definable functions $s_{p_{i}}$.
Let $C_{1},..,C_{r}$ be the cells of maximal dimension $n$. So by Claim 4, each $p_{i}$ of maximal $\partial$-dimension $n$ is in some $C_{j}$, and also $s_{p_{i}}$ is a $C^{1}$-function on $C_{j}$. 
Moreover each $C_{j}$ contains one of the types $p_{i}$ of maximal $\partial$-dimension.

Let us fix for each $j=1,..,r$, some $p_{i}$ which is in $C_{j}$. So $s_{p_{i}}$ is a $C^{1}$-function on $C_{j}$, and we rewrite $s_{p_{i}}$ as $s_{j}$. 

We conclude
\newline
{\em Claim 5.} ${\bar a}\in \Gamma$ is $\partial$-generic in $\Gamma$ over $k$ iff for some $j=1,..,r$, ${\bar a}\in C_{j}$, $dim_{L}({\bar a}/k) = n$ and $\partial({\bar a}) = s_{j}({\bar a})$. 

\vspace{2mm}
\noindent
Fix $j\in\{1,..,r\}$. As $C_{j}$ is an $n$-dimensional cell, let $\pi^{j}$ be a suitable projection from the ambient space 
to $M^{n}$ giving a bijection between $C_{j}$ and some open $L$-definable over $k$ subset $U_{j}$ of $M^{n}$  
(i.e., such that the map $(\pi^{j})^{-1}|U_{j}$ is $C^{1}$). 
By reordering coordinates, for ${\bar a} \in U_{j}$, write $\pi^{j}({\bar a})$ as $(a_{1},..,a_{n})$. We let 
$\pi^{j}(s_{j}): U_{j}\to M^{n}$ denote the $C^{1}$ function on $U_{j}$ taking $(a_{1},..,a_{n})$ to  the first $n$-coordinates of $s_{j}({\bar a})$ where ${\bar a}\in C_{j}$ and $\pi^{j}({\bar a}) = (a_{1},.,a_{n})$. 
With this notation, for  ${\bar a}\in C_{j}$, $s_{j}({\bar a})$ is determined by $\pi(s_{j})(a_{1},..,a_{n})$, using the 
formula in Remark 6.5.  We now identify  $C_{j}$ with $U_{j}$ via $\pi^{j}$ and also identify $s_{j}$ with 
$\pi^{j}(s_{j}):U_{j}\to M^{n}$.
 We conclude:
\newline
{\em Claim 6.} Let $U$ be the disjoint union of the $U_{j}$, and let $s$ be the disjoint union of the $s_{j}$. So (with the identification of $C_{j}$ and $U_{j}$ as above) the $\partial$-generic over $k$ elements of $\Gamma$ are precisely the elements ${\bar a}\in U$ such that $dim_{L}({\bar a}/k) = n$ and $\partial({\bar a}) = s({\bar a})$. 

\vspace{5mm}
\noindent
Both $*$ and $inv$ from Claim 3 induce partial $L_{k}$-definable operations on $U\times U$ and $U$ respectively, which we still call $*$ and $inv$. 

Note that, by the definition of $\partial$-independence we have that:
\newline
{\em Claim 7.} If $g_{1}, g_{2}\in \Gamma$ are $\partial$-generic  over $k$ and $\partial$-independent over $k$ then the product $g_{1}g_{2}$ is $\partial$-generic in $\Gamma$ over $k$ and $\partial$-independent from each of $g_{1}, g_{2}$ over $k$.  Likewise $g_{1}^{-1}\in \Gamma$ is $\partial$-generic over $k$.

Already Claims 6 and 7 give part of the current analogue of Lemma 3.4. The rest is given by:
\newline
{\em Claim 8.}  Let $a,b\in U$  be such that $dim_{L}(a,b/k) = 2n$ (so $a,b$ are $L$-independent $L$-generics of $U$ over $k$). Then there are  $\partial$-generic and $\partial$-independent over $k$, $g_{1}, g_{2}
\in \Gamma$, such that $tp_{L}(a,b)/k) = tp_{L}(g_{1}, g_{2}/k)$. 
\newline
\begin{proof} Suppose $a\in C_{j}$ and $b\in C_{\ell}$. By Lemma 6.7, there is $(g_{1}, g_{2})$ realizing $tp_{L}(a,b)/k)$ with $\partial(g_{1}) = s_{j}(g_{1})$ and $\partial(g_{2}) = s_{\ell}(g_{2})$. So $(g_{1}, g_{2})\in \Gamma\times\Gamma$, and $g_{1}$, $g_{2}$ are $\partial$-generic, $\partial$-independent over $k$  (as these properties depend only on $tp_{L}(g_{1}, g_{2}/k)$ once we know $g_{1}, g_{2}\in \Gamma$). 
\end{proof}

From Claims 6,7,8, we can prove the  analogue of Proposition 3.5 in our current situation. The construction of an enveloping $L$-definable group $G$ proceeds as in Section 4, or is given directly by Section 5.1.  The definable embeddability of $\Gamma$ in $G$ is as in Lemma 4.7.
This completes the
proof of part (i) of Proposition 6.8.

For the rest, by Theorem 5.4, the image of $U$ in $G$ is large, so by Lemma 6.7 and Claim 6 we get (ii) and (iii). 

\subsection{Open theories of topological fields with a generic derivation}
This is a context developed by the third author and collaborators for which the reader is  referred to \cite{G-P} and \cite{K-P}. It subsumes some new examples such as algebraically closed and real closed valued fields.

We first give a brief description of the relevant theories and then proceed to the main theorem and proof in this new context. It will be closer to the material in Section 2, 3,  4 and Section 5.2 than Section 6.1. 

We first fix a one-sorted language $L$, an extension of the language of rings by relation symbols and constant symbols. Let $T$ be a complete $L$-theory extending the theory of large fields of characteristic $0$ such that:
\newline
(1) For some $L$-formula $\phi(x,{\bar y})$, in any model $K$ of $T$, the set of formulas $\phi(x,{\bar a})$ as ${\bar a}$ ranges over tuples from $K$ of the right length, defines a basis for a (Hausdorff) topology on $K$ which makes $K$ into a topological field. When we say open, closed, etc we mean with respect to this topology.
\newline
(2)  For any model $K$ of $T$, and subfield $k$, any $k$-definable subset of $K^{n}$ is a finite union of sets of the form  $Z\cap U$ where $Z$ is a Zariski closed subset of $K^{n}$ defined by polynomials over $k$ and $U$ is an open subset of $K^{n}$ also defined over $k$. 


A certain many-sorted version of this setup-and the assumptions (1), (2),   with a distinguished topological field sort,
was defined in \cite{K-P}, and called an ``open theory of topological fields".  This many-sorted context was developed  so as to include  henselian valued fields of characteristic zero. But already the one-sorted context includes the complete theories $ACVF$ and $RCVF$ (treated as one-sorted structures).  In any case we will use the expression ``open theory of topological fields"  to refer to theories of (enriched) fields we consider, and the reader is free to restrict attention to the one-sorted context.   

We fix such an open theory of (large) topological fields $T$. There is no harm assuming quantifier elimination although this already holds in $ACVF$, $RCVF$ in the natural languages.

It is important to note that  $T$ is a ``geometric theory" (its models are geometric structures), by virtue of definability of dimension.  And $T$ is ``almost" a theory of geometric fields, except that the language $L$ may contain more than the ring language.

 We will typically denote models of $T$ by $K$ (although the language $L$ may be an expansion of the ring language), and a saturated model by ${\mathcal U}$. 
Nevertheless we see that in models of $T$, model-theoretic algebraic closure coincides with field-theoretic algebraic closure.  The dimension we use in models of $T$ will of course be algebraic dimension $dim({\bar a}/k) = trdeg(k({\bar a})/k)$.  Notice that by our assumptions on $T$, especially (2) above we have:
\begin{Remark} Let $K\models T$, $k$ a subfield and ${\bar a}$ an $n$-tuple from $K$. Then $tp({\bar a}/k)$ is 
axiomatized by  the following information ${\bar x}\in V$, where $V$ is the variety over $k$ with $k$-generic point ${\bar a}$, and $\{{\bar x}\in U\}$ where $U$ ranges over all open $k$-definable subsets of $V(K)$ containing ${\bar a}$. 
\end{Remark}

We have independence,  genericity, etc. in models of $T$, coming from $dim$. So for example, if $K\models T$, and $V$ is an irreducible variety over $k\leq K$ with a Zariski-dense set of $K$-points, then $tp({\bar a}/k)$ is a generic type of $V(K)$ if $dim({\bar a}/k) = dim(V)$. 

In our account of the main theorem below (Proposition 6.12) in the context of models of $T$ with a generic derivation, the topology on models of $T$ will not play much of a role. 

We  have the field-theoretic algebraic closure operation on $K$ giving rise to dimension, independence, on top of which is some additional structure which does not affect dimension, algebraic closure etc. 

It is proved in \cite{K-P} and \cite{G-P}, under additional assumptions to get consistency, that the class of existentially closed models of $T\cup\{\partial$ is a derivation\} is elementary with additional axioms: If $p(x_{0},..., x_{n})$ is a polynomial over $k$ (without loss, irreducible) and ${\bar a}$ is a $K$-point on the hypersurface $H_{p}$ defined by $p({\bar x}) = 0$ such that 
$(\partial p/\partial x_{n}) ({\bar a}) \neq 0$, then  for any open neighbourhood $U$ of ${\bar a}$ in $K^{n+1}$ there is $b\in K$ such that $\nabla^{(n)}(b)\in U\cap H_{p}$. 
If $T$ is a complete theory of topological fields for which the topology is induced by an henselian valuation (not necessarily definable) then $T_{\partial}$ is consistent. See Section 2.3 of \cite{K-P} where this is discussed. 
Here $\nabla^{(n)}$ is as we defined earlier.

We call this theory $T_{\partial}$. $T_{\partial}$ is a complete theory in $L_{\partial}$ which has QE if $T$ does (which we will assume). 

To prove our main theorem on finite-dimensional definable groups in this context (namely a finite-dimensional definable group $\Gamma$  in a model $(K,\partial)$ of $T_{\partial}$ definably embeds in a group $G$ definable in the $L$-structure $K$) we need as before a couple of additional facts:
\begin{Fact} (Lemma 3.1.9 of \cite{K-P}).  In models of $T_{\partial}$ the algebraic closure of a set $A$ coincides with the field-theoretic algebraic closure (equals $L$-algebraic closure) of the differential field generated by $A$, and the definable closure of a set $A$ coincides with the $L$-definable closure of the differential field generated by $A$.
\end{Fact}

For the next lemma we make use of rational $D$-varieties as defined above.
From now on we let $({\cal U},\partial)$ be a saturated model of $T_{\partial}$. $k, K,..$ will typically denote small differential subfields. 

\begin{Lemma} Let $(V,s)$ be a rational $D$-variety over $k$ ($V$ affine, irreducible, with a Zariski-dense set of ${\cal U}$-points) over $k$. Let 
$p({\bar x})$ be a complete generic type of $V$ over $k$,  in the sense of ${\mathcal U}\models T$. Then there is a realization  ${\bar a}$ of $p$ such that $\nabla({\bar a}) = s({\bar a})$.
\end{Lemma}
\begin{proof} 
The proof is like that of  Lemma 2.7.  Let ${\bar c}$ realize $p$ in some model $K$ of $T$ containing $k$. So ${\bar c}$ is a generic point of $V$ over $k$.  Apply $s$ to ${\bar c}$ to obtain $({\bar c}, {\bar d})\in T_{\partial}(V)$. The extension theorem for derivations allows us to extend $\partial|k$ to a derivation $\partial^{*}$ of $k({\bar c})$ with $\partial^{*}({\bar c}) = {\bar d}$.  By quantifier elimination of $T_{\partial}$ there is an embedding of the $L_{\partial}$-structure $(k({\bar c}),\partial^{*})$ into $({\mathcal U},\partial)$ over $(k,\partial)$. Let ${\bar a}$ be the image of ${\bar c}$. 

\end{proof}

All the earlier notions of finite-dimensional, $\partial$-independence, $\partial$-generics go over to the new context.

Again we have:
\begin{Theorem} Let $\Gamma$ be a finite-dimensional  group definable over $k$ in $({\mathcal U},\partial)\models T_{\partial}$ where $T$ is an open theory of large topological fields. Then $\Gamma$ is definably (over $k$) embeddable  in a group $G$ definable over $k$ in the $L$-structure ${\mathcal U}$. Remark 4.8 also holds. 
\end{Theorem}
\noindent
{\em Discussion.}   The proof is identical  to that of Theorem 4.1 (or 5.11) going through the same steps, except we replace ``semialgebraic" everywhere by $L$-definable (where we of course still have $dim_{L}$).  
\end{proof}


\begin{thebibliography}{99}

\bibitem{Baro-Martin-Pizarro} E. Baro and A. Martin-Pizarro, Open core and small groups in dense pairs of topological structures, Annals of Pure and Applied Logic, 
vol. 172 (2021). 


\bibitem{BCR} J. Bochnak, M. Coste, M-F Roy, Real Algebraic Geometry,  Springer, 1998. 

\bibitem{BCPP} Q. Brouette, G. Cousins, A. Pillay, F. Point, Embedded Picard-Vessiot extensions, Communications in Algebra 46 (2018), 4609 - 4615.

\bibitem{Bouscaren}  E. Bouscaren (ed.), Model Theory and Algebraic Geometry, Lecture Notes in Math 1696, Springer 1998.

\bibitem{Buium} A. Buium, Differential Algebraic Groups of Finite Dimension, Lecture Notes in Math. 1506,  Springer, 1992. 


\bibitem{CHvdP} T. Crespo, Z. Hajto, M, van der Put, Real and p-adic Picard-Vessiot fields, Math. Annalen, vol. 365 (2016), 93 - 103. 

\bibitem{Pantelis} P. Eleftheriou,  Pillay's conjecture for groups definable in weakly o-minimal non-valuational structures, Bulletin LMS 53 (2021), 1205 - 1219. 

\bibitem{F-K} A. Fornasiero and E. Kaplan, Generic derivations on $o$-minimal structures, Journal of Math. Logic, vol. 21
(2021). 

\bibitem{G-P} N. Guzy and F. Point, Topological differential fields, APAL 161 (2010), 570 - 598.



\bibitem{Hrushovski-locallymodular} E. Hrushovski,  Locally modular  regular types, in Classification Theory, Proceedings, Chicago 1985, Lecture Notes in Mathematics, Springer, 1987. 

\bibitem{Hrushovski-PAC} E. Hrushovski, Pseudofinite fields and related structure, Quaderni Math. vol. 11, Aracne, Rome 2002. 

\bibitem{Hrushovski-Pillay-groupslocalfields} E. Hrushovski and A. Pillay, Groups definable in local fields and pseudofinite fields, Israel Journal of Math., 85 (1994), 203 - 262. 

\bibitem{Metastable} E. Hrushovski and S. Rideau-Kikuchi, Valued fields, metastable groups, Selecta Mathematica vol. 25  (2019), paper no. 47, 58 pages. 



\bibitem{JY} W. Johnson and J. Ye, A note on geometric theories of fields, arXiv:2208.00586v1, august 1 2022.


\bibitem{K-P} P.C. Kovacsis and F. Point, Topological fields with a generic derivation, to appear in Annals of Pure and Applied Logic. 

\bibitem{Leon-Sanchez-Pillay}  O. Leon Sanchez and A. Pillay, Differential Galois cohomology and parameterized Picard-Vessiot extensions, Communications in Contemporary Mathematics, Vol. 23, No. 8 (2021).

\bibitem{Marker-Manin} D. Marker, Manin kernels,  p. 1- 21, Quad. Math.  vol 6, Dept. Math., Seconda Univ., Napoli, Caserta,  2000. 

\bibitem{Pierce-Pillay} D. Pierce and A. Pillay, A note on the axioms for differentially closed fields of characteristic zero, J. Algebra, 204 (1998), 108 - 115. 

\bibitem{Pillay-groupsandfields}  A. Pillay, On groups and fields definable in $o$-minimal structures, Journal Pure and Applied Algebra 53 (1988), 239 - 255.

\bibitem{Pillay-differential-algebraic-groups} A. Pillay, Differential algebraic groups and the number of countable differentially closed fields, in Model Theory of Fields, 2nd edition (edited, Marker, Messmer, Pillay) , Lecture Notes in Logic 4, ASL-CUP

\bibitem{Pillay-foundational} A. Pillay, Some foundational questions concerning differential algebraic groups. Pacific J. Math. 179 (1997), 179 - 200.

\bibitem{Pillay-DGTIV} A. Pillay, Algebraic $D$-groups and differential Galois theory, Pacific J. Math., 216 (2004), 343 - 360.

\bibitem{Pop} F. Pop, Little survey on large fields, {\em Valuation theory in interaction} 432 - 463, EMS Ser. Congr. Rep., Eur. Math. Soc., Z\"urich, 2014. 

\bibitem{Silvain-SCF}  S. Rideau-Kikuchi, A short note on groups in separably closed fields, Annals of Pure and Applied Logic 172 (2021), paper number 102943, 8 pages. 


\bibitem{Singer}  M. Singer, The model theory of ordered differential fields, JSL, 43 (1978), 82 - 91.


\bibitem{Tressl} M. Tressl, The uniform companion for large differential fields of characteristic $0$, TAMS, 357 (2005), 3933 - 3951.

\bibitem{vdDries} L. van den Dries, Weil's group chunk theorem: a topological setting, Illinois J. Math. 34 (1990), 127 - 139.


\bibitem{Weil} A. Weil, On algebraic groups of transformations, American J. Math. 77 (1955), 355 - 391.
\end{thebibliography}
\end{document}